\documentclass[11pt]{amsart}
\usepackage{amsmath,amsthm,amsfonts,amssymb,mathrsfs, mathtools}
\usepackage{enumerate}
\usepackage[colorlinks=true, linkcolor=blue, citecolor=magenta, menucolor=black]{hyperref}
\usepackage[left=3.5cm, right=3.5cm]{geometry}
\usepackage[utf8]{inputenc}
\usepackage[demo]{graphicx}
\usepackage[up]{caption}
\usepackage{subcaption}
\usepackage{pgf,tikz}
\usepackage[toc,page]{appendix}
\usetikzlibrary{arrows}
\usetikzlibrary{patterns}
\usepackage{color}

\makeatletter
\@namedef{subjclassname@2020}{%
  \textup{2020} Mathematics Subject Classification}
\makeatother

\usepackage{tikz-cd}

\numberwithin{equation}{section}

\newtheorem{letterthm}{Theorem}

\newtheorem{thm}{Theorem}[section]
\newtheorem{lem}[thm]{Lemma}

\newtheorem{cor}[thm]{Corollary}
\newtheorem{prop}[thm]{Proposition}

\theoremstyle{definition}
\newtheorem{rem}[thm]{Remark}

\newtheorem{df}[thm]{Definition}

\newtheorem{claim}[thm]{Claim}

\newcommand{\C}{\mathbb{C}}
\newcommand{\Z}{\mathbb{Z}}
\newcommand{\F}{\mathbb{F}}

\newcommand{\N}{\mathbb{N}}

\newcommand{\Tr}{\operatorname{Tr}}

\newcommand{\Id}{\operatorname{Id}}

\newcommand{\dpr}{^{\prime\prime}}

\newcommand{\rE}{\operatorname{E}}
\newcommand{\rC}{\operatorname{C}}
\newcommand{\rL}{\operatorname{L}}

\newcommand{\PSL}{\operatorname{PSL}}

\allowdisplaybreaks

\begin{document}

\title[Asymptotic freeness in tracial ultraproducts]{Asymptotic freeness in tracial ultraproducts}

\begin{abstract}
We prove novel asymptotic freeness results in tracial ultraproduct von Neumann algebras. In particular, we show that whenever $M = M_1 \ast M_2$ is a tracial free product von Neumann algebra and $u_1 \in \mathscr U(M_1)$, $u_2 \in \mathscr U(M_2)$ are Haar unitaries, the relative commutants $\{u_1\}' \cap M^{\mathcal U}$ and $\{u_2\}' \cap M^{\mathcal U}$ are freely independent in the ultraproduct  $M^{\mathcal U}$. Our proof relies on Mei--Ricard's results \cite{MR16} regarding $\rL^p$-boundedness (for all $1 < p < +\infty$) of certain Fourier multipliers in tracial amalgamated free products von Neumann algebras. We derive two applications. Firstly, we obtain a general absorption result in tracial amalgamated free products that recovers several previous maximal amenability/Gamma absorption results. 
Secondly, we prove a new lifting theorem which we combine with our asymptotic freeness results and Chifan--Ioana--Kunnawalkam Elayavalli's recent construction \cite{CIKE22} to provide the first example of a ${\rm II_1}$ factor that does not have property Gamma and is not elementary equivalent to any free product of diffuse tracial von Neumann algebras.
\end{abstract}

\author{Cyril Houdayer}
\address{\'Ecole Normale Sup\'erieure \\ D\'epartement de Math\'ematiques et Applications \\ Universit\'e Paris-Saclay \\ 45 rue d'Ulm \\ 75230 Paris Cedex 05 \\ FRANCE}
\email{cyril.houdayer@ens.psl.eu}
\thanks{CH is supported by Institut Universitaire de France}

\author{Adrian Ioana}
\address{Department of Mathematics \\ University of California San Diego \\ 9500 Gilman Drive \\ La Jolla \\ CA
92093 \\ USA}
\email{aioana@ucsd.edu}
\thanks{AI is supported by NSF DMS grants 1854074 and 2153805, and a Simons Fellowship}

\subjclass[2020]{46L10, 46L51, 46L54, 46L52, 03C66}
\keywords{Amalgamated free products; Continuous model theory; Noncommutative $\rL^p$-spaces; Ultraproducts; von Neumann algebras}

\maketitle

\section{Introduction}

In order to state our main results, we recall the following terminology regarding $n$-{\em independence} and {\em freeness}. 

\subsection*{Terminology} Let $(M, \tau)$ be a tracial von Neumann algebra together with a von Neumann subalgebra $B \subset M$. We denote by $\rE_B : M \to B$ the unique trace-preserving faithful normal conditional expectation and  set $M \ominus B = \ker(\rE_B)$.

Let $n \geq 1$. Following Popa (see e.g.\! \cite{Po13a, Po13b}), we say that two subsets $X, Y \subset M \ominus B$ are $n$-{\em independent} in $M$ with respect to $\rE_{B}$ if $\rE_B(x_1y_1\cdots x_ky_k)=0$, for every $1\leq k\leq n$, $x_1,\dots ,x_k\in X$ and $y_1,\dots,y_k\in Y$. We then say that two intermediate von Neumann subalgebras $B \subset M_1, M_2\subset M$ are $n$-{\em independent} in $M$ with respect to $\rE_B$ if the sets $M_1 \ominus B$ and $M_2 \ominus B$ are $n$-independent with respect to $\rE_B$. When $B = \C1$, two von Neumann subalgebras $M_1, M_2 \subset M$ are $1$-independent in $M$ with respect to $\tau$ if and only if $M_1$ and $M_2$ are $\tau$-orthogonal, i.e., $\tau(xy)=\tau(x)\tau(y)$, for every $x\in M_1$ and $y\in M_2$.

Let $I$ be a nonempty index set. We say that a family $(X_i)_{i \in I}$ of subsets of $M \ominus B$ is {\em freely independent} in $M$ with respect to $\rE_B$ if $\rE_B(x_1 \cdots x_k)=0$ for every $k \geq 1$, $x_1 \in X_{\varepsilon_1}, \dots, x_k \in X_{\varepsilon_k}$ with $\varepsilon_1 \neq \cdots \neq \varepsilon_k$ in $I$. We say that a family $(M_i)_{i \in I}$ of intermediate von Neumann subalgebras $B \subset M_i \subset M$ is {\em freely independent} in $M$ with respect to $\rE_B$ if the family of subsets $(M_i \ominus B)_{i \in I}$ is freely independent in $M$ with respect to $\rE_B$. In this case, we denote by $\ast_{B, i \in I} (M_i, \tau_i) = \bigvee_{i \in I} M_i \subset M$ the tracial amalgamated free product von Neumann algebra where $\tau_i = \tau|_{M_i}$ for every $i \in I$.

\subsection*{Main results}
Let $I$ be an at most countable index set such that $2 \leq | I | \leq +\infty$. Let $(M_i, \tau_i)_{i \in I}$ be a family of tracial von Neumann algebras with a common von Neumann subalgebra $(B, \tau_0)$ such that for every $i \in I$, we have $\tau_i|_B = \tau_0$. Denote by $(M, \tau) = \ast_{B, i \in I} (M_i, \tau_i)$ the tracial amalgamated free product von Neumann algebra. Let $\mathcal U$ be a nonprincipal ultrafilter on $\N$. Denote by $(M^{\mathcal U}, \tau^{\mathcal U})$ the tracial ultraproduct von Neumann algebra. Simply denote by $\rE_B : M \to B$ (resp.\! $\rE_{B^{\mathcal U}} : M^{\mathcal U} \to B^{\mathcal U}$) the unique trace-preserving faithful normal conditional expectation.

Denote by $\mathscr W \subset M$ the linear span of $B$ and of all the reduced words in $M$ of the form $w = w_1 \cdots w_n$, with $n \geq 1$, $w_j \in M_{\varepsilon_j} \ominus B$ for every $j \in \{1, \dots, n\}$, and $\varepsilon_1, \dots, \varepsilon_n \in I$ such that $\varepsilon_1 \neq \cdots \neq \varepsilon_n$. For every $i \in I$, denote by $\mathscr W_i \subset M \ominus B$  the linear span of all the reduced words in $M$ whose first and last letter lie in $M_i \ominus B$. Moreover, denote by $P_{\mathscr W_i} : \rL^2(M) \to \rL^2(\mathscr W_i)$ the corresponding orthogonal projection. By construction, the family $(\mathscr W_i)_{i \in I}$ is freely independent in $M$ with respect to $\rE_B$. Our first main result is an extension of this fact to the ultraproduct framework.

\begin{letterthm}\label{thm-main-result}
Keep the same notation as above. For every $i \in I$, denote by $\mathbf X_i$ the set of all the elements $x = (x_n)^{\mathcal U} \in M^{\mathcal U} \ominus B^{\mathcal U}$ such that $\lim_{n \to \mathcal U} \|x_n - P_{\mathscr W_i}(x_n)\|_2 = 0$. 

Then the family $(\mathbf X_i)_{i \in I}$ is freely independent in $M^{\mathcal U}$ with respect to $\rE_{B^{\mathcal U}}$. 
\end{letterthm}

The proof of Theorem \ref{thm-main-result} relies on Mei--Ricard's results \cite{MR16} showing that the canonical projection $P_{\mathscr W_i} : \mathscr W \to \mathscr W_i$ extends to a completely bounded operator $P_{\mathscr W_i} : \rL^p(M) \to \rL^p(\mathscr W_i)$ for every $p \in (1, +\infty)$. In particular, we exploit $\rL^p$-boundedness of $P_{\mathscr W_i} : \rL^p(M) \to \rL^p(\mathscr W_i)$ for every $2 \leq p < +\infty$. Theorem \ref{thm-main-result} is a novel application of noncommutative $\rL^p$-spaces to the structure theory of tracial von Neumann algebras.

It follows from Popa's asymptotic orthogonality property \cite{Po83} (see Lemma \ref{lem-commutant} below) that for every $i \in I$ and every unitary $u \in \mathscr U(M_i^{\mathcal U})$ such that $\rE_{B^{\mathcal U}}(u^k) = 0$ for every $k \in \Z \setminus \{0\}$, if $x \in \{u\}' \cap M^{\mathcal U}$ and $\rE_{B^{\mathcal U}}(x) = 0$, then $x \in \mathbf X_i$. In particular, in the case $B = \C 1$, Theorem \ref{thm-main-result} implies the following

\begin{letterthm}\label{cor-rel-comm}
    Assume that $B= \C 1$. For every $i\in I$, let $u_i \in \mathscr U(M_i^{\mathcal U})$ be a Haar unitary.
    
    Then the family $(\{u_i\}' \cap M^{\mathcal U})_{i\in I}$ is freely independent in $M^{\mathcal U}$ with respect to $\tau^{\mathcal U}$.
\end{letterthm}

Assume that for every $i \in I$,  $M_i$ is a diffuse abelian von Neumann algebra so that $M_i^{\mathcal U} \subset \{u_i\}' \cap M^{\mathcal U}$. Then Theorem \ref{cor-rel-comm} can be regarded as a far-reaching generalization of the fact that the family $(M_i^{\mathcal U})_{i \in I}$ is freely independent in $M^{\mathcal U}$ with respect to $\tau^{\mathcal U}$.

In the case $I = \{1, 2\}$, we also obtain the following variation of Theorem \ref{thm-main-result}. 

\begin{letterthm}\label{thm-main-result-bis}
Assume that $I = \{1, 2\}$. Keep the same notation as above. Let $\mathbf Y_1 \subset \mathbf X_1$ be a subset with the property that $a \mathbf Y_1 b \subset \mathbf X_1$ for all $a, b \in M_1$.  

Then the sets $\mathbf Y_1$ and $M \ominus M_1$ are freely independent in $M^{\mathcal U}$ with respect to $\rE_{B^{\mathcal U}}$.
\end{letterthm}

A typical example of a subset $\mathbf Y_1 \subset \mathbf X_1$ with the property that $a \mathbf Y_1 b \subset \mathbf X_1$ or all $a, b \in M_1$ is given by $\mathbf Y_1 = A' \cap (M^{\mathcal U} \ominus M_1^{\mathcal U})$, where $A \subset M_1$ is a von Neumann subalgebra such that $A \npreceq_{M_1} B$ (see Lemmas \ref{lem-intertwining} and \ref{lem-commutant-bis}).

In the case $M = B \rtimes \F_n = (B \rtimes \Z) \ast_B \cdots \ast_B (B \rtimes \Z) = M_1 \ast_B \cdots \ast_B M_n$, Popa showed in \cite[Lemma 2.1]{Po83} that for the canonical Haar unitary $u \in \rL(\Z) \subset M_1$, the sets $\{u\}' \cap (M^{\mathcal U} \ominus M_1^{\mathcal U})$ and $M \ominus M_1$ are $2$-independent in $M^{\mathcal U}$ with respect to $\rE_{B^{\mathcal U}}$ (see also \cite{HU15} for the free product case). Letting $\mathbf Y_1 = \{u\}' \cap (M^{\mathcal U} \ominus M_1^{\mathcal U})$, Theorem \ref{thm-main-result-bis} can be regarded as a generalization and a strengthening of Popa's result.

In the case $I = \{1, 2\}$ and $B = \C 1$, we exploit Mei--Ricard's results \cite{MR16} to obtain the following indecomposability result in $M^{\mathcal U}$.

\begin{letterthm}\label{31infinity}
Assume that $I=\{1,2\}$ and $B=\mathbb C1$. Let $u_1\in\mathscr U(M_1^{\mathcal U})$ be a Haar unitary and $u_2\in\mathscr U(M_2^{\mathcal U})$ such that $\tau^{\mathcal U}(u_2)=\tau^{\mathcal U}(u_2^2)=0$.

Then there do not exist $v_1,v_2\in\mathscr U(M^{\mathcal U})$ such that $\tau^{\mathcal U}(v_1)=\tau^{\mathcal U}(v_1^2)=\tau^{\mathcal U}(v_2)=0$ and
 $[u_1,v_1]=[v_1,v_2]=[v_2,u_2]=0$. 
 \end{letterthm}
 
Another way to formulate Theorem \ref{31infinity} is as follows. Let $u_1\in\mathscr U(M_1^{\mathcal U})$ be a Haar unitary, $u_2\in\mathscr U(M_2^{\mathcal U})$ such that $\tau^{\mathcal U}(u_2)=\tau^{\mathcal U}(u_2^2)=0$ and $v_1\in\mathscr U(M^{\mathcal U})$ such that $[u_1, v_1] = 0$ and $\tau^{\mathcal U}(v_1)=\tau^{\mathcal U}(v_1^2)= 0$. Then we have $\{v_1, u_2\}' \cap M^{\mathcal U} = \C 1$. 
This generalizes the well-known fact (see e.g.\! \cite[Lemma 6.1]{Io12}) that $\{u_1,u_2\}' \cap M^{\mathcal U} = \C 1$.
In Section \ref{section-main-results}, we generalize Theorem \ref{31infinity} to arbitrary tracial amalgamated free product von Neumann algebras (see Theorem \ref{3infinity}).

Theorem \ref{31infinity} is new even in the case $M=\rL(C_1*C_2)$, where $C_1, C_2$ are cyclic groups with $|C_1|>1$ and $|C_2|>2$. In this case, $M=M_1 \ast M_2$, where $M_1=\rL(C_1), M_2=\rL(C_2)$.
Moreover, $M$ is an interpolated free group factor by \cite[Corollary 5.3]{Dy92}  and thus has positive $1$-bounded entropy, $h(M)>0$, in the sense of \cite{Ju05,Ha15}. 
By  \cite[Corollary 4.8]{Ha15} (see also \cite[Facts 2.4 and 2.9]{CIKE22}), if $u_1,u_2\in M$ are generating unitaries, then there are no Haar unitaries $v_1,v_2\in M^{\mathcal U}$ satisfying $[u_1,v_1]=[v_1,v_2]=[v_2,u_2]=0$. This fact was used in \cite{CIKE22} to construct two non-elementarily equivalent non-Gamma ${\rm II_1}$ factors.

Theorem \ref{31infinity} considerably strengthens this fact when $C_1=\mathbb Z$, $u_1\in \mathscr U(M_1)$, $u_2\in \mathscr U(M_2)$.   Unlike \cite{Ha15}, we cannot say anything about arbitrary generating unitaries $u_1$ and $u_2$, that do not belong to $M_1$ and $M_2$, respectively. On the other hand, while the free entropy methods from  \cite{Ha15} 
only rule out the existence of Haar unitaries $v_1,v_2 \in \mathscr U(M^{\mathcal U})$ satisfying $[u_1,v_1]=[v_1,v_2]=[v_2,u_2]=0$, Theorem \ref{31infinity} also excludes the existence of such finite order unitaries $v_1,v_2$ provided that $v_1,v_1^2,v_2$ have trace zero. 

Let $u_1\in\mathscr U(M_1^{\mathcal U})$ and $u_2\in\mathscr U(M_2^{\mathcal U})$ be as in Theorem \ref{31infinity}, and assume that $u_2^m=1$, for some $m>2$. Then $\{u_2\}'\cap M^{\mathcal U}$ has finite index in $M^{\mathcal U}$ and therefore, unlike in Theorem \ref{cor-rel-comm}, $\{u_1\}'\cap M^{\mathcal U}$ and $\{u_2\}'\cap M^{\mathcal U}$ are not freely independent in $M^{\mathcal U}$ with respect to $\tau^{\mathcal U}$. Instead, the proof of Theorem \ref{31infinity} relies on a subtler analysis of commuting unitaries belonging to $\{u_1\}'\cap M^{\mathcal U}$ and $\{u_2\}'\cap M^{\mathcal U}$. However, similarly to the proof of Theorem \ref{cor-rel-comm}, we make crucial use of Mei--Ricard's results \cite{MR16}.

We do not know if Theorem \ref{31infinity} holds if we remove the assumption that $\rE_{B^{\mathcal U}}(u_2^2)=0$. However, a standard diagonal argument implies the existence of $N\in\mathbb N$ such  that the assumption that  $\rE_{B^{\mathcal U}}(u_1^k)=0$, for every $k\in\mathbb Z\setminus\{0\}$, can be relaxed by assuming instead that $\rE_{B^{\mathcal U}}(u_1^k)=0$, for every $k\in\mathbb Z\setminus\{0\}$ with $|k|\leq N$.

\subsection*{Application to absorption in AFP von Neumann algebras}

We use Theorem \ref{thm-main-result-bis} to obtain a new absorption result for tracial amalgamated free product von Neumann algebras.

\begin{letterthm}\label{cor-amenable-absorption}
Assume that $I = \{1, 2\}$. Keep the same notation as above and assume that $M$ is separable. Let $P \subset M$ be a von Neumann subalgebra such that $P \cap M_1 \npreceq_{M_1} B$ and $P' \cap M^{\mathcal U} \npreceq_{M^{\mathcal U}} B^{\mathcal U}$. Then we have $P \subset M_1$. 
\end{letterthm} 

Theorem \ref{cor-amenable-absorption} vastly generalizes Popa's seminal result \cite{Po83} that the generator masa $\text{L}(\langle a\rangle)$ is maximal amenable inside the free group factor $\text{L}(\mathbb F_2)=\text{L}(\langle a,b\rangle)$.  Specifically, it  extends several maximal amenability/Gamma absorption results.
 Theorem \ref{cor-amenable-absorption} generalizes \cite[Theorem A]{HU15} (see also \cite[Theorem A]{Ho14}) to arbitrary tracial amalgamated free product von Neumann algebras. As we observe in Remark \ref{rem-amenable}, if $P \subset M$ is an amenable von Neumann subalgebra such that $P \cap M_1 \npreceq_{M_1} B$, then we necessarily have $P' \cap M^{\mathcal U} \npreceq_{M^{\mathcal U}} B^{\mathcal U}$. Thus, Theorem \ref{cor-amenable-absorption} also yields a new proof of \cite[Main theorem]{BH16} in the setting of tracial amalgamated free product von Neumann algebras.
 
Let us point out that in the setting of tracial free products $M = M_1 \ast M_2$ of Connes-embeddable von Neumann algebras, the inclusion $M_1 \subset M$ satisfies a more general absorption property. Indeed, \cite[Theorem A]{HJNS19} shows that if $P \subset M$ is a von Neumann subalgebra such that $P \cap M_1$ is diffuse and has $1$-bounded entropy zero, then $P \subset M_1$. In the case $M = \rL(\F_n)$ is a free group factor, the aforementioned absorption property holds for {\em any} diffuse maximal amenable subalgebra $Q \subset M$, thanks to the recent resolution of the Peterson--Thom conjecture via random matrix theory \cite{BC22, BC23} and $1$-bounded entropy \cite{Ha15} (see also \cite{HJKE23}).

\subsection*{Application to continuous model theory of ${\rm II_1}$ factors}

We next present an application of Theorem \ref{cor-rel-comm} to the continuous model theory of ${\rm II_1}$ factors. A main theme in this theory,  initiated by Farah--Sherman--Hart in \cite{FHS11}, is to determine whether two given ${\rm II_1}$ factors $M, N$ are elementarily equivalent. By the continuous version of the Keisler--Shelah theorem this amounts to $M$ and  $N$ admitting isomorphic ultrapowers, $M^{\mathcal U}\cong N^{\mathcal V}$, for some ultrafilters $\mathcal U$ and $\mathcal V$ on arbitrary sets \cite{FHS11,HI02}.
It was shown in \cite{FHS11} that property Gamma and being McDuff are elementary properties, leading to three distinct elementary classes of ${\rm II_1}$ factors.
A fourth such elementary class was then provided in \cite{GH16}. The problem of determining the number of elementary classes of ${\rm II_1}$ factors was solved in \cite{BCI15}, where the continuum of non-isomorphic ${\rm II_1}$ factors constructed in \cite{Mc69} were shown to be pairwise non-elementarily equivalent. However, all the available techniques for distinguishing ${\rm II_1}$ factors up to elementary equivalence were based on central sequences. It thus remained open to construct any non-elementarily equivalent ${\rm II_1}$ factors which do not have any non-trivial central sequences, i.e., fail property Gamma.

This problem was solved by Chifan--Ioana--Kunnawalkam Elayavalli~in \cite{CIKE22} using a combination of techniques from Popa's deformation/rigidity theory and Voiculescu's free entropy theory. 
First, deformation/rigidity methods from \cite{IPP05} were used to construct a non-Gamma ${\rm II_1}$ factor $M$ via an iterative amalgamated free product construction. It was then shown that $M$ is not elementarily equivalent to any (necessarily non-Gamma) ${\rm II_1}$ factor $N$ having positive $1$-bounded entropy, $h(N)>0$, in the sense of Jung \cite{Ju05} and Hayes \cite{Ha15}.
Examples of ${\rm II_1}$ factors $N$ with $h(N)>0$ include the interpolated free group factors $\rL(\mathbb F_t)$, $1< t\leq\infty$, and, more generally, any tracial free product $N=N_1 \ast N_2$ of  diffuse Connes-embeddable von Neumann algebras. 
For additional examples of such ${\rm II}_1$ factors, see \cite[Fact 2.7]{CIKE22}.
However, the methods from \cite{CIKE22} could not distinguish $M$ up to elementary equivalence from  $N=N_1*N_2$, whenever $N_1$ or $N_2$ is a non-Connes-embeddable tracial von Neumann algebra (the existence of which has been announced in the preprint \cite{JNVWY20}). 
 In particular, since it is unclear if $M$ is Connes-embeddable, it remained open whether $M$ is elementarily equivalent to $M*\rL(\mathbb Z)$. 

Theorem \ref{cor-rel-comm} allows us to settle this problem for a variant of the ${\rm II_1}$ factor constructed in \cite{CIKE22}:
 
\begin{letterthm}\label{non-ee}
There exists a separable ${\rm II_1}$ factor $M$ which does not have property Gamma and that is not elementarily equivalent to $N=N_1*N_2$, for any diffuse tracial von Neumann algebras $(N_1,\tau_1)$ and $(N_2,\tau_2)$.
\end{letterthm}

In particular, Theorem \ref{non-ee} provides the first example of a non-Gamma ${\rm II_1}$ factor $M$ which is not elementarily equivalent to $M*\rL(\mathbb Z)$.
The conclusion of Theorem \ref{non-ee} is verified by any ${\rm II_1}$ factor $M$ satisfying the following:

\begin{letterthm}\label{inductive}
There exists a separable ${\rm II_1}$ factor $M$ which does not have property Gamma and satisfies the following. For every countably cofinal ultrafilter $\mathcal U$ on a set $J$ and $u_1,u_2\in\mathscr U(M^{\mathcal U})$ such that $\{u_1\}\dpr$ and $\{u_2\}\dpr$ are $2$-independent in $M^{\mathcal U}$ with respect to $\tau^{\mathcal U}$, there exist Haar unitaries $v_1,v_2\in\mathscr U(M^{\mathcal U})$ such that $[u_1,v_1]=[u_2,v_2]=[v_1,v_2]=0$.
\end{letterthm}

An ultrafilter $\mathcal U$ on a set $J$ is called {\em countably cofinal} if there exists a sequence $(A_n)_{n\in\mathbb N}$ in $\mathcal U$ with $\bigcap_{n\in\mathbb N}A_n=\emptyset$. Any free ultrafilter on $\mathbb N$ is countably cofinal.

The proof of Theorem \ref{inductive} uses the iterative amalgamated free product construction introduced in \cite{CIKE22}. In \cite[Theorem B]{CIKE22}, this construction was used to build a non-Gamma separable ${\rm II_1}$ factor $M$ with the following property: for any  unitaries $u_1,u_2\in \mathscr U(M^{\mathcal U})$ such that  $\{u_1\}\dpr$ and $\{u_2\}\dpr$ are orthogonal and $u_1^2=u_2^3=1$, there exist Haar unitaries $v_1,v_2\in\mathscr U(M^{\mathcal U})$ such that $[u_1,v_1]=[u_2,v_2]=[v_1,v_2]=0$. The proof of \cite[Theorem B]{CIKE22} relies crucially on a lifting lemma showing that any unitaries $u_1,u_2\in \mathscr U(M^{\mathcal U})$ such that  $\{u_1\}\dpr$ and $\{u_2\}\dpr$ are orthogonal and $u_1^2=u_2^3=1$ lift to unitaries in $M$ with the same properties. A key limitation in \cite{CIKE22} was the assumption that $u_1$ and $u_2$ have orders $2$ and $3$.  We remove this limitation here by proving a general lifting result  of independent interest (see Theorem \ref{lifting}) which shows that any unitaries $u_1,u_2\in \mathscr U(M^{\mathcal U})$ such that $\{u_1\}\dpr$ and $\{u_2\}\dpr$ are $2$-independent admit lifts $u_1=(u_{1,n})^{\mathcal U}$ and $u_2=(u_{2,n})^{\mathcal U}$ with $\{u_{1,n}\}\dpr$ and $\{u_{2,n}\}\dpr$ orthogonal for every $n \in \N$. 
With this result in hand, adjusting the iterative construction from \cite{CIKE22} implies Theorem \ref{inductive}. 

To explain how Theorem \ref{non-ee} follows by combining Theorem \ref{inductive} and Theorem \ref{cor-rel-comm}, let $M$ be a ${\rm II_1}$ factor as in Theorem \ref{inductive}, $N=N_1 \ast N_2$ a free product of diffuse tracial von Neumann algebras and $u_1\in \mathscr U(N_1)$, $u_2\in \mathscr U(N_2)$ Haar unitaries.  
Since $\{u_1\}\dpr$ and $\{u_2\}\dpr$ are freely and thus $2$-independent, it follows that
$M^{\mathcal U}\not\cong N^{\mathcal V}$, for any countably cofinal ultrafilter $\mathcal U$ and any ultrafilter $\mathcal V$.
Indeed, Theorem \ref{cor-rel-comm} implies that any Haar unitaries (more generally, any trace zero unitaries) $v_1,v_2\in\mathscr U(N^{\mathcal V})$ such that $[u_1,v_1]=[u_2,v_2]=0$ are freely independent and therefore do not commute. 
If $\mathcal U$ is an ultrafilter which is not countably cofinal, then we also have that $M^{\mathcal U}\not\cong N^{\mathcal V}$, for any ultrafilter $\mathcal V$. Otherwise,  using \cite[Lemma 2.3]{BCI15} we would get that $M^{\mathcal U}\cong M$, thus $N^{\mathcal V}\cong M$ is separable, hence $N^{\mathcal V}\cong N$ and therefore $M\cong N$. But then  $M^{\mathcal W}\cong N^{\mathcal W}$, for any free ultrafilter $\mathcal W$ on $\mathbb N$. Since $\mathcal W$ is countably cofinal, this is a contradiction. Altogether, we conclude that $M^{\mathcal U}\not\cong N^{\mathcal V}$, for any ultrafilters $\mathcal U,\mathcal V$, and thus $M, N$ are not elementarily equivalent.

\subsection*{Application to the orthogonalization problem} We end the introduction with an application to the following orthogonalization problem: given a ${\rm II}_1$ factor $M$  and two subsets $X,Y\subset M\ominus\mathbb C1$, when can we find $u\in\mathscr U(M)$ such that $uXu^*$ and $Y$ are orthogonal?
 In the case $X=A\ominus\mathbb C1$ and $Y=B\ominus \mathbb C1$, for von Neumann subalgebras $A,B\subset M$, this and related independence problems have been studied extensively by Popa (see e.g.\! \cite{Po13a,Po13b,Po17}). When $X,Y\subset M\ominus\mathbb C1$ are finite, a standard averaging argument shows that
 we can find $u\in\mathscr U(M)$ such that $uXu^*$ and $Y$ are ``almost orthogonal":
  for every $\varepsilon>0$, there exists $u\in\mathscr U(M)$ such that $|\langle uxu^*,y\rangle|<\varepsilon$, for every $x\in X,y\in Y$. This implies the existence of $v\in\mathscr U(M^{\mathcal U})$, where $\mathcal U$ is a free ultrafilter on $\mathbb N$, such that $vXv^*$ and $Y$ are orthogonal. In this context, much more is true: by a result of Popa (see \cite[Corollary 0.2]{Po13a}),  if $X,Y\subset M\ominus \mathbb C1$ are countable, then there exists $u\in\mathscr U(M^{\mathcal U})$ such that $uXu^*$ and $Y$ are freely independent.

By combining this result with the proof of our lifting theorem (Theorem \ref{lifting})  we settle affirmatively the above orthogonalization problem whenever $X,Y\subset M\ominus\mathbb C1$ are finite.

\begin{letterthm}\label{orthogonal}
Let $M$ be a ${\rm II}_1$ factor and $X,Y\subset M\ominus\mathbb C1$ be finite sets.

Then there exists $u\in\mathscr U(M)$ such that $uXu^*$ and $Y$ are orthogonal.
\end{letterthm}

\subsection*{Acknowledgements} This work was initiated when CH was visiting the University of California at San Diego (UCSD) in March 2023. He thanks the Department of Mathematics at UCSD for its kind hospitality. The authors thank Ben Hayes, Srivatsav Kunnawalkam Elayavalli and Sorin Popa for their useful comments.



\section{Preliminaries}

\subsection{Noncommutative $\rL^p$-spaces}

Let $(M, \tau)$ be a tracial von Neumann algebra. For every $p \in [1, +\infty)$, we write $\rL^p(M) = \rL^p(M, \tau)$ for the completion of $M$ with respect to the norm $\|\cdot\|_p$ defined by $\|x\|_p = \tau(|x|^p)^{1/p}$ for every $x \in M$. 
More generally, given a subspace $\mathscr W\subset M$, we denote by $\rL^p(\mathscr W)\subset\rL^p(M)$ the closure of $\mathscr W$ with respect to $\|\cdot\|_p$.
Then $\rL^p(M)$ is the noncommutative $\rL^p$-space associated with the tracial von Neumann algebra $M$. We simply write $\rL^\infty(M) = M$. 

We will use the following generalized noncommutative H\"older inequality (see e.g.\! \cite[Theorem IX.2.13]{Ta03}): for all $k \geq 2$, all $p_1, \dots, p_k, r \in [1, +\infty)$ such that $\frac1r = \sum_{j = 1}^k \frac{1}{p_j}$ and all $(x_j)_j \in \prod_{j = 1}^k \rL^{p_j}(M)$, we have $x = x_1 \cdots x_k \in \rL^r(M)$ and $\|x_1 \cdots x_k\|_r \leq \|x_1\|_{p_1} \cdots \|x_k\|_{p_k}$. 

For all $1 \leq p \leq q < +\infty$ and all $x \in M$, we have $\|x\|_1 \leq \|x\|_p \leq \|x\|_q \leq \|x\|_\infty$ and so we may regard $M \subset \rL^q(M) \subset \rL^p(M) \subset \rL^1(M)$. 

\subsection{Ultraproduct von Neumann algebras}

Let $\mathcal U$ be a nonprincipal ultrafilter on $\N$. Whenever $(M, \tau)$ is a tracial von Neumann algebra, we denote by $(M^{\mathcal U}, \tau^{\mathcal U})$ the tracial ultraproduct von Neumann algebra. We regard $\rL^2(M^{\mathcal U}) \subset \rL^2(M)^{\mathcal U}$ as a closed subspace and we denote by $e : \rL^2(M)^{\mathcal U} \to \rL^2(M^{\mathcal U})$ the corresponding orthogonal projection. Recall the following elementary yet useful facts.

\begin{lem}\label{lem-ultraproduct}
Keep the same notation as above. The following assertions hold:
\begin{itemize}
\item [$(\rm i)$] Let $(\xi_n)_n$ be a $\|\cdot\|_2$-bounded sequence in $\rL^2(M)$ and set $\xi = (\xi_n)^{\mathcal U} \in \rL^2(M)^{\mathcal U}$. Then $\xi \in \rL^2(M^{\mathcal U})$ if and only if for every $\varepsilon > 0$, there exists a $\|\cdot\|_\infty$-bounded sequence $ (x_n)_n$  in $M$ such that $\lim_{n \to \mathcal U} \|\xi_n - x_n \|_2 \leq \varepsilon$.

\item [$(\rm ii)$] Let $r > 2$. Then for every $\|\cdot\|_r$-bounded sequence $(\xi_n)_n$ in $\rL^r(M)$, we have $\xi = (\xi_n)^{\mathcal U} \in \rL^2(M^{\mathcal U})$.

\item [$(\rm iii)$] Let $(\xi_n)_n$ be a $\|\cdot\|_2$-bounded sequence in $\rL^2(M)$. Let $(x_n)_n$ and $(y_n)_n$ be $\|\cdot\|_\infty$-bounded sequences in $M$. Set $\xi = (\xi_n)^{\mathcal U} \in \rL^2(M)^{\mathcal U}$, $x = (x_n)^{\mathcal U} \in M^{\mathcal U}$ and $y = (y_n)^{\mathcal U} \in M^{\mathcal U}$. If $\xi \in \rL^2(M^{\mathcal U})$, then $(x_n \xi_n y_n)^{\mathcal U} = x \xi y \in \rL^2(M^{\mathcal U})$.
\end{itemize}
\end{lem}

\begin{proof}
$(\rm i)$ It is straightforward.

$(\rm ii)$ Let $(\xi_n)_n$ be a $\|\cdot\|_r$-bounded sequence in $\rL^r(M)$. There exists $\kappa > 0$ such that $\sup_{n \in \N} \tau(|\xi_n|^r) < \kappa$. Set $\xi = (\xi_n)^{\mathcal U} \in \rL^2(M)^{\mathcal U}$. For every $n \in \N$, write $\xi_n = v_n |\xi_n|$ for the polar decomposition of $\xi_n \in \rL^r(M)$. For every $n \in \N$ and every $t > 0$, define the spectral projection $p_{n, t} = \mathbf 1_{[0, t]}(|\xi_n|) \in M$. For every $n \in \N$ and every $t > 0$, we have 
$$\|\xi_n \, p_{n, t}^\perp\|_2^2 \leq \| |\xi_n| \, p_{n, t}^\perp \|_2^2 = \tau(|\xi_n|^2 \mathbf 1_{(t, +\infty)}(|\xi_n|)) \leq \frac{1}{t^{r - 2}} \tau(|\xi_n|^r \mathbf 1_{(t, +\infty)}(|\xi_n|)) \leq \frac{\kappa}{t^{r - 2}}.$$
Let $\varepsilon > 0$ and choose $t > 0$ large enough so that $\frac{\kappa}{t^{r - 2}} \leq \varepsilon^2$. For every $n \in \N$, set $x_n = \xi_n \, p_{n, t}\in M$ and observe that we have $\|\xi_n - x_n\|_2 = \|\xi_n \, p_{n, t}^\perp\|_2 \leq \varepsilon$. Since $\sup \left\{ \|x_n\|_\infty \mid n \in \N \right \} \leq t$, Item $(\rm i)$ implies that $\xi = (\xi_n)^{\mathcal U} \in \rL^2(M^{\mathcal U})$.

$(\rm iii)$ Assume that $\xi = (\xi_n)^{\mathcal U} \in \rL^2(M^{\mathcal U})$. Choose $\kappa > 0$ large enough so that $\sup \left\{ \|x_n\|_\infty, \|y_n\|_\infty \mid n \in \N \right\} \leq \kappa$. Let $\varepsilon > 0$. By Item $(\rm i)$, there exists a $\|\cdot\|_\infty$-bounded sequence $(z_n)_n$ in $M$ such that $\lim_{n \to \mathcal U} \|\xi_n - z_n\|_2 \leq \varepsilon$. Set $z = (z_n)^{\mathcal U} \in M^{\mathcal U}$. Then $\|\xi - z\|_2 = \lim_{n \to \mathcal U} \|\xi_n - z_n\|_2 \leq \varepsilon$. Since $x z y = (x_n z_n y_n)^{\mathcal U} \in M^{\mathcal U}$, we have
\begin{align*}
\|(x_n \xi_n y_n)^{\mathcal U} - x \xi y\|_2 &\leq \|(x_n \xi_n y_n)^{\mathcal U} -  (x_n z_n y_n)^{\mathcal U} \|_2 + \| x z y - x \xi y\|_2 \\
&= \lim_{n \to \mathcal U} \|x_n (\xi_n - z_n) y_n\|_2 + \| x (z  -  \xi) y\|_2 \\
&\leq 2 \kappa^2 \|\xi - z\|_2 \leq 2 \kappa^2 \varepsilon.
\end{align*}
Since this holds for every $\varepsilon > 0$, it follows that $(x_n \xi_n y_n)^{\mathcal U} = x \xi y \in \rL^2(M^{\mathcal U})$.
\end{proof}

We also record the following basic fact concerning tracial ultraproducts:
\begin{lem}\label{p-norm}
Keep the same notation as above. Let $(x_n)_n$ be a $\|\cdot\|_\infty$-bounded sequence in $M$. Set $x=(x_n)^{\mathcal U}\in M^{\mathcal U}$. Then for every $p\in [1,+\infty)$ we have  $\|x\|_p=\lim_{n\rightarrow\mathcal U}\|x_n\|_p$. 
\end{lem}

\begin{proof}
We may assume that $\sup \left\{\|x_n\|_\infty\mid n\in\mathbb N\right\}\leq 1$ and thus $\|x\|_\infty\leq 1$. Note that $|x|^2=x^*x=(x_n^*x_n)^{\mathcal U}=(|x_n|^2)^{\mathcal U}$. Then for every $k\in\mathbb N$, we have $|x|^{2k}=(|x_n|^{2k})^{\mathcal U}$ and thus $\tau^{\mathcal U}(|x|^{2k})=\lim_{n\rightarrow\mathcal U}\tau(|x_n|^{2k})$. Thus, if $P(t)\in\mathbb C[t]$ is a polynomial with complex coefficients and $Q(t)=P(t^2)$, then $\tau^{\mathcal U}(Q(|x|))=\lim_{n\rightarrow\mathcal U}\tau(Q(|x_n|))$. Since by the Stone--Weierstrass theorem $\{P(t^2)\mid P(t)\in\mathbb C[t]\}$ is dense in $\rC([0,1])$ in the uniform norm, we get that $\tau^{\mathcal U}(f(|x|))=\lim_{n\rightarrow\mathcal U}\tau(f(|x_n|))$, for every $f\in \rC([0,1])$. In particular, $\tau^{\mathcal U}(|x|^p)=\lim_{n\rightarrow\mathcal U}\tau(|x_n|^p)$, for every $p\in [1,+\infty)$, which implies the conclusion. 
\end{proof}

\subsection{Amalgamated free products}\label{subsection-amalgamated}

Let $I$ be an at most countable index set such that $2 \leq | I | \leq +\infty$. Let $(M_i, \tau_i)_{i \in I}$ be a family of tracial von Neumann algebras with a common von Neumann subalgebra $(B, \tau_0)$ such that for every $i \in I$, we have $\tau_i|_B = \tau_0$. Denote by $(M, \tau) = \ast_{B, i \in I} (M_i, \tau_i)$ the tracial amalgamated free product von Neumann algebra. Simply denote by $\rE_B : M \to B$ the unique trace-preserving faithful normal conditional expectation.

Denote by $\mathscr W \subset M$ the linear span of $B$ and of all the reduced words in $M$ of the form $w = w_1 \cdots w_n$, with $n \geq 1$, $w_j \in M_{\varepsilon_j} \ominus B$ for every $j \in \{1, \dots, n\}$, and $\varepsilon_1, \dots, \varepsilon_n \in I$ such that $\varepsilon_1 \neq \cdots \neq \varepsilon_n$. For every subset $J \subset I$, denote by $\mathscr L_J \subset \mathscr W$ (resp.\! $\mathscr R_J \subset \mathscr W$) the linear span of all the reduced words whose first (resp.\! last) letter lies in $M_j \ominus B$ for some $j \in J$. For every $i \in I$, denote by $\mathscr W_i \subset \mathscr W$  the linear span of all the reduced words whose first and last letter lie in $M_i \ominus B$. We will use the following consequences of Mei--Ricard's results (see \cite[Theorem 3.5]{MR16}).

\begin{thm}[Mei--Ricard \cite{MR16}]\label{thm-Mei-Ricard}
Let $p \in (1, +\infty)$, $J \subset I$, and $i \in I$. The following assertions hold:
\begin{itemize}
\item [$(\rm i)$] The  projection map $P_{\mathscr L_J} : \mathscr W \to \mathscr L_J$  extends to a completely bounded operator $P_{\mathscr L_J} : \rL^p(M) \to \rL^p(\mathscr L_J)$. 

\item [$(\rm ii)$] The projection map $P_{\mathscr R_J} : \mathscr W \to \mathscr R_J$ extends to a completely bounded operator $P_{\mathscr R_J} : \rL^p(M) \to \rL^p(\mathscr R_J)$.

\item [$(\rm iii)$] The projection map $P_{\mathscr W_i} : \mathscr W \to \mathscr W_i$ extends to a completely bounded operator $P_{\mathscr W_i} : \rL^p(M) \to \rL^p(\mathscr W_i)$.
\end{itemize}
\end{thm}

\begin{proof}
We use the notation $H_\varepsilon$ of \cite[Section 3]{MR16}.

$(\rm i)$ For every $j \in J$, set $\varepsilon_j = - 1$ and for every $j \in I\setminus J$, set $\varepsilon_j = 1$. Then with $\varepsilon = (\varepsilon_i)_{i \in I}$, we have $P_{\mathscr L_J} = \frac{\Id - H_\varepsilon}{2}$. Therefore, \cite[Theorem 3.5]{MR16} implies that $P_{\mathscr L_J} : \mathscr W \to \mathscr L_J$ extends to a completely bounded operator $P_{\mathscr L_J} : \rL^p(M) \to \rL^p(\mathscr L_J)$. 

$(\rm ii)$ The proof  is completely analogous to Item $(\rm i)$.

$(\rm iii)$ We have $P_{\mathscr W_i} = P_{\mathscr L_{i}} \circ P_{\mathscr R_{i}} = P_{\mathscr R_{i}} \circ P_{\mathscr L_{i}}$. Therefore, Items $(\rm i)$ and $(\rm ii)$ imply that $P_{\mathscr W_i} : \mathscr W \to \mathscr W_i$ extends to a completely bounded operator $P_{\mathscr W_i} : \rL^p(M) \to \rL^p(\mathscr W_i)$.
\end{proof}

Let $P_p$ be one of the operators from Theorem \ref{thm-Mei-Ricard} (i.e., $P_{\mathscr L_J}, P_{\mathscr R_J}$, or $P_{\mathscr W_i}$). Then the operators $(P_p)_{p\in (1,+\infty)}$ are consistent with the inclusions $\rL^q(M)\subset\rL^p(M)$,  in the sense that $P_q=P_p |_{\rL^q(M)}$, for every $1<p\leq q<+\infty$. This is why we denote $P$ instead of $P_p$.

\subsection{Popa's intertwining theory}

We review Popa's criterion for intertwining von Neumann subalgebras \cite{Po01, Po03}. Let $(M, \tau)$ be a tracial von Neumann algebra and $A\subset 1_A M 1_A$, $B \subset 1_B M 1_B$ be von Neumann subalgebras. By \cite[Corollary 2.3]{Po03} and \cite[Theorem A.1]{Po01} (see also \cite[Proposition C.1]{Va06}), the following conditions are equivalent:

\begin{itemize}
\item [$(\rm i)$] There exist $n \geq 1$, a projection $q \in \mathbf M_n(B)$, a nonzero partial isometry $v \in \mathbf M_{1, n}(1_A M)q$ and a unital normal $\ast$-homomorphism $\pi : A \to q\mathbf M_n(B)q$  such that $a v = v \pi(a)$ for all $a \in A$.

\item [$(\rm ii)$] There exist projections $p \in A$ and $q \in B$, a nonzero partial isometry $v \in pMq$ and a unital normal $\ast$-homomorphism $\pi : pAp \to qBq$ such that $a v = v \pi(a)$ for all $a \in A$.

\item [$(\rm iii)$] There is no net of unitaries $(w_k)_k$ in $A$ such that
$$\forall x, y \in 1_A M 1_B, \quad \lim_k \|\rE_B(x^* w_k y)\|_2 = 0.$$
\end{itemize}

If one of the previous equivalent conditions is satisfied, we say that $A$ {\it embeds into} $B$ {\it inside} $M$ and write $A \preceq_M B$.

Following \cite{Jo82,PP84}, we say that an inclusion of tracial von Neumann algebras $P \subset M$ has {\em finite index} if $\rL^2(M, \tau)$ has finite dimension as a right $P$-module. If $A_0 \subset A$ is a von Neumann subalgebra with finite index and if $A \preceq_M B$, then $A_0 \preceq_M B$ (see \cite[Lemma 3.9]{Va07}).

We record the following new criterion for intertwining von Neumann subalgebras.

\begin{lem}\label{lem-intertwining}
Let  $(M,\tau)$ be a separable tracial von Neumann algebra and $A,B\subset M$ be von Neumann subalgebras such that $A\npreceq_MB$.
Then there exists $u\in\mathscr U(A^{\mathcal U})$ such that $\rE_{B^{\mathcal U}}(xu^my)=0$, for all $x,y\in M$ and all $m\in\mathbb Z\setminus\{0\}$.
\end{lem}

\begin{proof}
To prove the lemma, it suffices to argue that for every finite subset $F\subset M$, $\varepsilon>0$ and $K\in\mathbb N$, we can find $u\in\mathscr U(A)$ such that $\|\rE_B(xu^my^*)\|_2<\varepsilon$, for all $m\in\mathbb Z\setminus\{0\}$ with $|m|\leq K$.  To this end, fix  a finite subset $F\subset M$, $\varepsilon>0$ and $K\in\mathbb N$. For $u\in\mathscr U(M)$, set $\psi(u)=\sum_{m\in \mathbb Z\setminus\{0\},|m|\leq K}\sum_{x,y\in F}\|\rE_B(xu^my^*)\|_2^2$.

Let $v\in\mathscr U(M)$  with $\{v\}\dpr \npreceq_MB$. For every $N\in\mathbb N$, set 
$$\varphi(v,N)=\frac{1}{N}\sum_{k=1}^N\sum_{x,y\in F}\|\rE_B(xv^ky^*)\|_2^2.$$ 
We claim that $\lim_{N \to \infty}\varphi(v,N)=0$. Indeed, set $\xi=\sum_{x\in F}xe_Bx^*\in\langle M,B\rangle$, where $(\langle M, B\rangle, \Tr)$ is Jones basic construction of $B\subset M$. Using that $\|\rE_B(z)\|_2^2=\text{Tr}(ze_Bz^*e_B)$ for every $z\in M$, we obtain that 
\begin{equation}\label{average}
\forall N \in \N, \quad \varphi(v,N)=\Tr \left( \left(\frac{1}{N}\sum_{k=1}^Nv^k\xi {v^{-k}} \right)\xi \right).
\end{equation}
By von Neumann's ergodic theorem, there exists $\eta\in\rL^2(\langle M,B\rangle, \Tr)$ such that $v\eta v^*=\eta$ and $\lim_{N \to \infty}\|\frac{1}{N}\sum_{k=1}^Nv^k\xi{v^{-k}}-\eta\|_{2, \Tr}=0$. Then $w\eta=\eta w$, for all $w\in\{v\}\dpr$. Since $\{v\}\dpr\npreceq_MB$, we obtain that $\eta=0$. In combination with \eqref{average}, this proves our claim that $\lim_{N\to\infty}\varphi(v,N)=0$.

We are now ready to finish the proof. Since $A\npreceq_MB$, we can find a diffuse abelian von Neumann subalgebra $A_0\subset A$ such that $A_0\npreceq_MB$ (see \cite[Corollary F.14]{BO08}).
Let $v\in\mathscr U(A_0)$ be a Haar unitary with $\{v\}\dpr=A_0$. If $m\in\mathbb Z\setminus\{0\}$, then $\{v^m\}\dpr\subset A_0$ has finite index, and thus $\{v^m\}\dpr\npreceq_MB$. The above claim gives that $\lim_{N \to \infty}\varphi(v^m,N)=0$, for all $m\in\mathbb Z\setminus\{0\}$. Thus, we can find $N\in\mathbb N$ such that $\sum_{m\in\mathbb Z\setminus\{0\},|m|\leq K}\varphi(v^m,N)<\varepsilon^2$. Since $\sum_{m\in\mathbb Z\setminus\{0\},|m|\leq K}\varphi(v^m,N)=\frac{1}{N}\sum_{k=1}^N\psi(v^k)$, we find $1\leq k\leq N$ such that $\psi(v^k)<\varepsilon^2$. Thus, $u=v^k$ satisfies the desired conclusion, which finishes the proof.
\end{proof}


\section{Proofs of Theorems \ref{thm-main-result}, \ref{cor-rel-comm}, \ref{thm-main-result-bis}, \ref{31infinity}}\label{section-main-results}

\subsection{Popa's asymptotic orthogonality property}

Let $I$ be an at most countable index set such that $2 \leq | I | \leq +\infty$. Let $(M_i, \tau_i)_{i \in I}$ be a family of tracial von Neumann algebras with a common von Neumann subalgebra $(B, \tau_0)$ such that for every $i \in I$, we have $\tau_i|_B = \tau_0$. Denote by $(M, \tau) = \ast_{B, i \in I} (M_i, \tau_i)$ the tracial amalgamated free product von Neumann algebra. Simply denote by $\rE_B : M \to B$ (resp.\! $\rE_{B^{\mathcal U}} : M^{\mathcal U} \to B^{\mathcal U}$) the unique trace-preserving faithful normal conditional expectation.

The following lemma is a generalization of Popa's asymptotic orthogonality property (see \cite[Lemma 2.1]{Po83}) in the framework of tracial amalgamated free product von Neumann algebras. The key new feature of the proof is that we exploit Theorem \ref{thm-Mei-Ricard} to work inside the Hilbert space $\rL^2(M^{\mathcal U})$ instead of $\rL^2(M)^{\mathcal U}$ as in Popa's proof.

\begin{lem}\label{lem-commutant}
Let $i \in I$. Let $u \in \mathscr U(M_i^{\mathcal U})$ be a unitary such that $\rE_{B^{\mathcal U}}(u^k) = 0$ for every $k \in \Z \setminus \{0\}$. For every $x = (x_n)^{\mathcal U} \in \{u\}' \cap M^{\mathcal U}$ such that $\rE_{B^{\mathcal U}}(x) = 0$, we have $\lim_{n \to \mathcal U} \|x_n - P_{\mathscr W_i}(x_n)\|_2 = 0$.
\end{lem}

\begin{proof}
Let $x = (x_n)^{\mathcal U} \in \{u\}' \cap M^{\mathcal U}$ such that $\rE_{B^{\mathcal U}}(x) = 0$. Without loss of generality, we may assume that $\|x_n\|_\infty \leq 1$ for every $n \in \N$. To prove that $\lim_{n \to \mathcal U} \|x_n - P_{\mathscr W_i}(x_n)\|_2 = 0$, we show that $\lim_{n \to \mathcal U} \|P_{\mathscr L_{I\setminus \{i\}}}(x_n)\|_2 = \lim_{n \to \mathcal U} \|P_{\mathscr R_{I\setminus \{i\}}}(x_n)\|_2 = 0$. Since $\mathscr R_{I \setminus \{i\}} = J \mathscr L_{I \setminus \{i\}} J$, it suffices to prove that $\lim_{n \to \mathcal U} \|P_{\mathscr L_{I\setminus \{i\}}}(x_n)\|_2 = 0$. To simplify the notation, we set $P_i = P_{\mathscr L_{I\setminus \{i\}}}$.

By Lemma \ref{lem-ultraproduct}$(\rm ii)$ and Theorem \ref{thm-Mei-Ricard}$(\rm i)$, we have $(P_i(x_n))^{\mathcal U} \in \rL^2(M^{\mathcal U})$. Set $\mathscr H_i = \rL^2(M^{\mathcal U}) \cap (\mathscr L_{I \setminus \{i\}})^{\mathcal U} \subset \rL^2(M^{\mathcal U})$ and denote by $P_{\mathscr H_i} : \rL^2(M^{\mathcal U}) \to \mathscr H_i$ the corresponding orthogonal projection. Then we have $P_{\mathscr H_i}(x) = (P_i(x_n))^{\mathcal U} \in \mathscr H_i$. 
For every $N \geq 1$, we have
\begin{align}\label{eq-parallelogram}
N \cdot \| P_{\mathscr H_i}(x)\|^2_2 &= \sum_{k =1}^N  \|u^k P_{\mathscr H_i}(x) u^{-k}\|^2_2 \\  \nonumber
&= \sum_{k =1}^N  \|P_{u^{k}\mathscr H_i u^{-k}}(u^k x u^{-k})\|^2_2 \\ \nonumber
&= \sum_{k =1}^N  \|P_{ u^{k}\mathscr H_i u^{-k}}(x)\|^2_2.
\end{align}
We claim that the Hilbert subspaces $( u^{k}\mathscr H_i u^{-k} )_{k \in \Z}$ are mutually orthogonal in $\rL^2(M^{\mathcal U})$ i.e.\! for every $k \in \Z \setminus \{0\}$, $ u^{k}\mathscr H_i u^{-k}$ and $ \mathscr H_i $ are orthogonal in $\rL^2(M^{\mathcal U})$. Indeed, for every $k \in \Z \setminus \{0\}$, since $\rE_{B^{\mathcal U}}(u^k) = 0$, we may write $u^k = (u_{n, k})^{\mathcal U} \in M_i^{\mathcal U}$ where $(u_{n, k})_{n}$ is a $\|\cdot\|_\infty$-bounded sequence in $M_i \ominus B$. Let $\xi = (\xi_n)^{\mathcal U} \in \mathscr H_i$ and $\eta = (\eta_n)^{\mathcal U} \in \mathscr H_i$ where $(\xi_n)_n$ and $(\eta_n)_n$ are $\|\cdot\|_2$-bounded sequences in $\mathscr L_{I \setminus \{i\}}$. By construction, it is plain to see that for all $n \in \N$ and all $k \in \Z \setminus \{0\}$, the vectors $u_{n, k}  \xi_n u_{n, k}^*$ and $\eta_n$ are orthogonal in $\rL^2(M)$. By Lemma \ref{lem-ultraproduct}$(\rm iii)$, since $\xi = (\xi_n)^{\mathcal U} \in \rL^2(M^{\mathcal U})$, we have 
$$\langle u^k \xi u^{-k} , \eta \rangle = \langle (u_{n, k}  \xi_n u_{n, k}^*)^{\mathcal U} , (\eta_n)^{\mathcal U} \rangle = \lim_{n \to \mathcal U} \langle u_{n, k} \xi_n u_{n, k}^* , \eta_n \rangle = 0.$$
This finishes the proof of the claim.

From \eqref{eq-parallelogram}, since the projections $(P_{u^k \mathscr H_i u^{-k}})_{k \in \Z}$ are mutually orthogonal, we infer that 
\begin{equation}\label{eq-conclusion}
N \cdot \| P_{\mathscr H_i}(x)\|^2_2 = \sum_{k =1}^N  \|P_{ u^{k}\mathscr H_i u^{-k}}(x)\|^2_2 \leq  \|x\|_2^2.
\end{equation}
Since \eqref{eq-conclusion} holds for every $N \geq 1$, we have $\lim_{n \to \mathcal U} \|P_i(x_n)\|_2 = \| P_{\mathscr H_i}(x) \|_2 = 0$. This finishes the proof of the lemma.
\end{proof}

Using a $2 \times 2$ matrix trick, we obtain the following extension of Lemma \ref{lem-commutant}.

\begin{lem}\label{lem-commutant-bis}
Let $i \in I$. Let $u \in \mathscr U(M_i^{\mathcal U})$ be a unitary such that $\rE_{B^{\mathcal U}}(v u^k v^*) = 0$ for every $k \in \Z \setminus \{0\}$ and every $v \in \mathscr U(M_i)$. For every $x = (x_n)^{\mathcal U} \in \{u\}' \cap M^{\mathcal U}$ such that $\rE_{M_i^{\mathcal U}}(x) = 0$ and every $y, z \in M_i$, we have $\lim_{n \to \mathcal U} \|y x_n z - P_{\mathscr W_i}(y x_n z)\|_2 = 0$.
\end{lem}

\begin{proof}
Set $\mathscr B = \mathbf M_2(B)$, $\mathscr M_j = \mathbf M_2(M_j)$ for every $j \in I$ and $\mathscr M = \mathbf M_2(M)$ so that we have $\mathscr M = \ast_{\mathscr B, j \in I} \mathscr M_j$. Let $x = (x_n)^{\mathcal U} \in \{u\}' \cap M^{\mathcal U}$ be such that $\rE_{M_i^{\mathcal U}}(x) = 0$. Since any element of $M_i$ is a linear combination of at most four unitaries of $M_i$, it suffices to prove that for every $v, w \in \mathscr U(M_i)$, we have $\lim_{n \to \mathcal U} \|v x_n w - P_{\mathscr W_i}(v x_n w)\|_2 = 0$. 

Let $v, w \in \mathscr U(M_i)$. Set 
$$U = \begin{pmatrix}
v uv^* & 0 \\
0 & w^* u w
\end{pmatrix} \in \mathscr U(\mathscr M_i^{\mathcal U})  \quad \text{and} \quad X = \begin{pmatrix}
0 & v x w \\
0 & 0
\end{pmatrix} \in \mathscr M^{\mathcal U} \ominus \mathscr M_i^{\mathcal U}.$$
By construction, we have $U X = X U$ and $\rE_{\mathscr B^{\mathcal U}}(U^k) = 0$ for every $k \in \Z \setminus \{0\}$. We may now apply Lemma \ref{lem-commutant} to $X = (X_n)^{\mathcal U} \in \mathscr M^{\mathcal U} \ominus \mathscr B^{\mathcal U}$ and conclude that 
$$\lim_{n \to \mathcal U} \|v x_n w - P_{\mathscr W_i}(v x_n w)\|_2 = \lim_{n \to \mathcal U} \| (X_n)_{12} - P_{\mathscr W_i}((X_n)_{12})\|_2 = 0.$$
This finishes the proof of the lemma.
\end{proof}

\subsection{Proofs of Theorem \ref{thm-main-result}, Theorem \ref{cor-rel-comm}, Theorem \ref{thm-main-result-bis}}

\begin{proof}[Proof of Theorem \ref{thm-main-result}]
Keep the same notation as in the statement of Theorem \ref{thm-main-result}. Let $k \geq 1$ and $\varepsilon_1, \dots, \varepsilon_k \in I$ be such that $\varepsilon_1 \neq \cdots \neq \varepsilon_k$. For every $1 \leq j \leq k$, let $x_j = (x_{j, n})^{\mathcal U} \in \mathbf X_{\varepsilon_j}$. We may assume that $\sup \left\{\|x_{j, n}\|_\infty \mid 1 \leq j \leq k, n \in \N \right\} \leq 1$. We show that $\rE_{B^{\mathcal U}}(x_1 \cdots x_k) = 0$. 

For every $i \in I$ and every $p \in (1, +\infty)$, we simply denote by $P_i : \rL^p(M) \to \rL^p(\mathscr W_i)$ the completely bounded operator (see Theorem \ref{thm-Mei-Ricard}$(\rm iii)$). For every $p \in (1, +\infty)$, choose $\kappa_p > 0$ large enough so that $$\sup \left\{\|P_i(x)\|_p \mid i \in \{\varepsilon_1, \dots, \varepsilon_k\}, x \in \rL^p(M), \|x\|_p \leq 1 \right\} \leq \kappa_p.$$
Fix $1 < r < 2$ (e.g.,\! $r = \frac32$). Let $p \in (1, +\infty)$ be such that $\frac1r = \frac12 + \frac{k - 1}{p}$. For every $1 \leq j \leq k$ and every $n \in \N$, write $x_{j, n} = P_{\varepsilon_j}(x_{j, n}) + (x_{j, n} - P_{\varepsilon_j}(x_{j, n}))$ and observe that 
\begin{itemize}
\item $P_{\varepsilon_j}(x_{j, n}) \in \rL^2(\mathscr W_{\varepsilon_j})$ and $\lim_{n \to \mathcal U} \|x_{j, n} - P_{\varepsilon_j}(x_{j, n})\|_2 = 0$;
\item $P_{\varepsilon_j}(x_{j, n}) \in \rL^p(\mathscr W_{\varepsilon_j})$ and $\max \left\{ \|P_{\varepsilon_j}(x_{j, n}) \|_{p}, \|x_{j, n} - P_{\varepsilon_j}(x_{j, n})\|_p \right\} \leq 1 + \kappa_p$. 
\end{itemize}
For every $n \in \N$, we may write $x_{1, n} \cdots x_{k, n} - P_{\varepsilon_1}(x_{1, n}) \cdots P_{\varepsilon_k}(x_{k, n})$ as a sum of $2^k - 1$ terms that are products of length $k$ for which at least one of the factors is of the form $x_{j, n} - P_{\varepsilon_j}(x_{j, n})$ for some $1 \leq j \leq k$. For every $n \in \N$, using the triangle inequality and the generalized noncommutative H\"older inequality, we obtain
$$
\|x_{1, n} \cdots x_{k, n} - P_{\varepsilon_1}(x_{1, n}) \cdots P_{\varepsilon_k}(x_{k, n})\|_r \leq (2^{k} - 1) (1 + \kappa_p)^{k - 1} \max_j \|x_{j, n} - P_{\varepsilon_j}(x_{j, n})\|_2.
$$
This implies that 
\begin{equation}\label{eq-1}
\lim_{n \to \mathcal U} \|x_{1, n} \cdots x_{k, n} - P_{\varepsilon_1}(x_{1, n}) \cdots P_{\varepsilon_k}(x_{k, n})\|_r = 0.
\end{equation}

Next, set $q = kr $ so that $\frac1r = \frac{k}{q}$. For every $1 \leq j \leq k$ and every $n \in \N$, since $P_{\varepsilon_j}(x_{j, n}) \in \rL^q(\mathscr W_{\varepsilon_j})$, we may choose $w_{j, n} \in \mathscr W_{\varepsilon_j}$ such that $\|P_{\varepsilon_j}(x_{j, n}) - w_{j, n}\|_q \leq \frac{1}{n + 1}$.  For every $1 \leq j \leq k$ and every $n \in \N$, write $P_{\varepsilon_j}(x_{j, n}) = w_{j, n} + (P_{\varepsilon_j}(x_{j, n}) - w_{j, n})$ and  observe that 
\begin{itemize}
\item $\max \left\{ \|w_{j, n} \|_{q}, \|P_{\varepsilon_j}(x_{j, n}) - w_{j, n} \|_q \right\} \leq 1 + \kappa_q$. 
\end{itemize}
We may then write $P_{\varepsilon_1}(x_{1, n}) \cdots P_{\varepsilon_k}(x_{k, n}) - w_{1, n} \cdots w_{k, n}$ as a sum of $2^k - 1$ terms that are products of length $k$ for which at least one of the factors is of the form $P_{\varepsilon_j}(x_{j, n}) - w_{j, n}$ for some $1 \leq j \leq k$. For every $n\in \N$, using the triangle inequality and the generalized noncommutative H\"older inequality, we obtain
$$
\|P_{\varepsilon_1}(x_{1, n}) \cdots P_{\varepsilon_k}(x_{k, n}) - w_{1, n} \cdots w_{k, n}\|_r \leq (2^{k} - 1) (1 + \kappa_q)^{k - 1} \max_j \|P_{\varepsilon_j}(x_{j, n}) - w_{j, n}\|_q.
$$
This implies that 
\begin{equation}\label{eq-2}
\lim_{n \to \mathcal U} \| P_{\varepsilon_1}(x_{1, n}) \cdots P_{\varepsilon_k}(x_{k, n}) - w_{1, n} \cdots w_{k, n}\|_r = 0.
\end{equation}

By combining \eqref{eq-1} and \eqref{eq-2}, it follows that 
\begin{equation}\label{eq-3}
\lim_{n \to \mathcal U} \| x_{1, n} \cdots x_{k, n} - w_{1, n} \cdots w_{k, n}\|_r = 0.
\end{equation}
In particular, since $\rE_{B^{\mathcal U}} (x_1 \cdots x_k)=(\rE_B (x_{1, n} \cdots x_{k, n}))^{\mathcal U}$, using Lemma \ref{p-norm} and the fact that $\rE_B$ is $\|\cdot\|_r$-contractive,
we have  
$$\|\rE_{B^{\mathcal U}} (x_1 \cdots x_k) \|_r = \lim_{n \to \mathcal U} \|\rE_B (x_{1, n} \cdots x_{k, n}) \|_r = \lim_{n \to \mathcal U} \|\rE_B (w_{1, n} \cdots w_{k, n}) \|_r.$$
Since for every $1 \leq j \leq k$ and every $n \in \N$, we have $w_{j, n} \in \mathscr W_{\varepsilon_j}$ and since $\varepsilon_1 \neq \cdots \neq  \varepsilon_k$, it follows that $\rE_B(w_{1, n} \cdots w_{k, n}) = 0$. Thus, we obtain $\rE_{B^{\mathcal U}}(x_1 \cdots x_k) =0$. 
\end{proof}

\begin{rem}
We were informed by Sorin Popa that he and Stefaan Vaes had recently made the following observation. In the case $\rL(\F_2) = A_1 \ast A_2$ is a free group factor with $A_1 \cong A_2 \cong \rL(\Z)$, they showed that $A_1' \cap \rL(\F_2)^{\mathcal U}$ and $A_2$ are freely independent in $\rL(\F_2)^{\mathcal U}$ with respect to $\tau^{\mathcal U}$.
\end{rem}

We obtain the following consequence of Theorem \ref{thm-main-result} which implies Theorem \ref{cor-rel-comm}.

\begin{thm}\label{cor-main-result}
Assume that $B = \C 1$. For every $i \in I$, let $(A_{i, k})_{k \in \N}$ be a decreasing sequence of separable diffuse abelian von Neumann subalgebras of $M_i^{\mathcal U}$ such that $\bigcap_{k = 1}^\infty A_{i, k} = \C 1$. 

Then for every $i \in I$, $\mathscr M_i = \bigvee_{k = 1}^\infty (A_{i, k}' \cap M^{\mathcal U}) \subset M^{\mathcal U}$ is a nonamenable irreducible subfactor with property Gamma. Moreover, the family $(\mathscr M_i)_{i \in I}$ is freely independent in $M^{\mathcal U}$ with respect to $\tau^{\mathcal U}$. 
\end{thm}

\begin{proof}
Let $i \in I$. For every $k \in \N$, since $A_{i, k}$ is separable, we have $(A_{i, k}' \cap M^{\mathcal U})' \cap M^{\mathcal U} = A_{i, k}$ by \cite[Theorem 2.1]{Po13a}. This further implies that 
$$\mathscr M_i' \cap M^{\mathcal U} = (\bigvee_{k = 1}^\infty A_{i, k}' \cap M^{\mathcal U})' \cap M^{\mathcal U} = \bigcap_{k = 1}^\infty  (A_{i, k}' \cap M^{\mathcal U})' \cap M^{\mathcal U} = \bigcap_{k = 1}^\infty A_{i, k} = \C 1.$$
Therefore, $\mathscr M_i \subset M^{\mathcal U}$ is a irreducible subfactor. Since $A_{i, 0}' \cap M^{\mathcal U}$ is nonamenable, $\mathscr M_i$ is nonamenable as well. Let $\mathscr V$ be another nonprincipal ultrafilter on $\N$. For every $k \in \N$ and every $\lambda \in (0, 1)$, choose a projection $p_{\lambda, k} \in A_{i, k}$ such that $\tau^{\mathcal U}(p_{\lambda, k}) = \lambda$. Then $p_{\lambda} = (p_{\lambda, k})^{\mathcal V} \in \mathscr M_i' \cap \mathscr M_i^{\mathcal V}$ is a projection such that $(\tau^{\mathcal U})^{\mathcal V}(p_\lambda) = \lambda$. Therefore, $\mathscr M_i' \cap \mathscr M_i^{\mathcal V}$ is a diffuse von Neumann algebra and so $\mathscr M_i$ has property Gamma.

A combination of Lemma \ref{lem-commutant} and Theorem \ref{thm-main-result} implies that all $k, \ell \in \N$, the family $(A_{i, k}' \cap M^{\mathcal U})_{i \in I}$ is freely independent in $M^{\mathcal U}$ with respect to $\tau^{\mathcal U}$. Since for every $i \in I$, the sequence of von Neumann subalgebras $(A_{i, k}' \cap M^{\mathcal U})_k$ is increasing, Kaplansky's density theorem further implies that the family $(\mathscr M_i)_{i \in I}$ is freely independent in $M^{\mathcal U}$ with respect to $\tau^{\mathcal U}$.
\end{proof}

\begin{proof}[Proof of Theorem \ref{thm-main-result-bis}]
Keep the same notation as in Theorem \ref{thm-main-result}. Let $\mathbf Y_1 \subset \mathbf X_1$ be a subset with the property that $a \mathbf Y_1 b \subset \mathbf X_1$ for all $a, b \in M_1$. Denote by $M_1 \mathbf Y_1 M_1$ the linear span of all the elements of the form $a Y b$ for $a, b \in M_1$ and $Y \in \mathbf Y_1$. Then we have $M_1 \mathbf Y_1 M_1 \subset \mathbf X_1$. Likewise, denote by $M_1 \mathscr W_2 M_1$ the linear span of all the elements of the form $a w b$ for $a, b \in M_1$ and $w \in \mathscr W_2$. Observe that any word with letters alternating from $\mathbf Y_1$ and $M_1 \mathscr W_2 M_1$ can be written as a linear combination of words with letters alternating from $M_1 \mathbf Y_1 M_1 \cup (M_1 \ominus B)$ and $\mathscr W_2$. Since $M_1 \mathbf Y_1 M_1 \cup (M_1 \ominus B) \subset \mathbf X_1$ and $\mathscr W_2 \subset \mathbf X_2$, Theorem \ref{thm-main-result} implies that the sets $\mathbf Y_1$ and $M_1 \mathscr W_2 M_1$ are freely independent in $M^{\mathcal U}$ with respect to $\rE_{B^{\mathcal U}}$. 

Using Kaplansky's density theorem, for any element $x \in M \ominus M_1$, there exists a $\|\cdot\|_\infty$-bounded sequence $(x_n)_n$ in $M_1 \mathscr W_2 M_1$ such that $x_n \to x$ for the strong operator topology. Combining this fact with the first paragraph of the proof, we infer that the sets $\mathbf Y_1$ and $M \ominus M_1$ are freely independent in $M^{\mathcal U}$ with respect to $\rE_{B^{\mathcal U}}$. 
\end{proof}

\subsection{Proof of Theorem \ref{31infinity}}
This subsection is devoted to the proof of Theorem \ref{31infinity}. Moreover, we generalize Theorem \ref{31infinity} to arbitrary tracial amalgamated free product von Neumann algebras.

\begin{thm}\label{3infinity}
Assume that $I=\{1,2\}$. Let $u_1\in\mathscr U(M_1^{\mathcal U})$ and $u_2\in\mathscr U(M_2^{\mathcal U})$ be such that $\rE_{B^{\mathcal U}}(u_1^k)=0$, for every $k\in\mathbb Z\setminus\{0\}$, and $\rE_{B^{\mathcal U}}(u_2)=\rE_{B^{\mathcal U}}(u_2^2)=0$.

Then there do not exist unitaries $v_1,v_2\in\mathscr U(M^{\mathcal U})$ such that $[u_1,v_1]=[v_1,v_2]=[v_2,u_2]=0$ and $\rE_{B^{\mathcal U}}(v_1)=\rE_{B^{\mathcal U}}(v_1^2)=\rE_{B^{\mathcal U}}(v_2)=0$.
\end{thm}

The proof of Theorem \ref{3infinity} relies on two lemmas. Using the notation from Section \ref{subsection-amalgamated}, for every $i,j\in I$, we put $\mathscr W_{i,j}=\mathscr L_i\cap\mathscr R_j$ and $P_{i,j}=P_{\mathscr L_i}\circ P_{\mathscr R_j}$.  By Theorem \ref{thm-Mei-Ricard}, we have a completely bounded operator $P_{i,j}:\rL^p(M)\to \rL^p(\mathscr W_{i,j})$, for every $p\in (1,+\infty)$.

\begin{lem}\label{almost-ortho}
Assume that $I=\{1,2\}$. Let $u \in \mathscr U(M_1^{\mathcal U})$ be such that $\rE_{B^{\mathcal U}}(u^k) = 0$ for all $k \in \Z \setminus \{0\}$. Let $x = (x_n)^{\mathcal U} \in \mathscr U(\{u\}' \cap M^{\mathcal U})$ and $y=(y_n)^{\mathcal U}\in \{x\}'\cap M^{\mathcal U}$ with  $\rE_{B^{\mathcal U}}(x)=\rE_{B^{\mathcal U}}(y)=0$. 
Then  $\lim_{n\rightarrow\mathcal U}\|P_{2,2}(y_n)\|_2^2=\lim_{n\rightarrow\mathcal U}\langle x_nP_{2,2}(y_n), P_{1,1}(y_n)x_n\rangle$. Thus, $\lim_{n\rightarrow \mathcal U}\|P_{2,2}(y_n)\|_2\leq\lim_{n\rightarrow\mathcal U}\|P_{1,1}(y_n)\|_2$.
\end{lem}

\begin{proof}
Since $\rE_{B^{\mathcal U}}(y)=0$, after replacing $y_n$ with $y_n-\rE_B(y_n)$, we may assume that $\rE_B(y_n)=0$, for all $n\in\mathbb N$. 
We may assume that $x_n\in\mathscr U(M)$ and $\|y_n\|_\infty\leq 1$, for all  $n\in\mathbb N$.
For $p\in (1,+\infty)$, let $\kappa_p=\sup \left\{\|P_{i,j}(x)\|_p\mid i,j\in \{1,2\}, x\in \rL^p(M),\|x\|_p\leq 1\right\}.$

For every $n\in\mathbb N$ and $i,j\in\{1,2\}$, put $y_n^{i,j}=P_{{i,j}}(y_n)$. Then $y_n=\sum_{i,j=1}^2y_n^{i,j}$ and $\|y_n^{i,j}\|_p\leq\kappa_p$, for every $n\in\mathbb N,i,j\in\{1,2\}$ and $p\in (1,+\infty)$. We claim that 
\begin{equation}\label{perp}\text{$\lim_{n\rightarrow\mathcal U}\|\rE_B(x_ny_n^{2,2}x_n^*{y_n^{i,j}}^*)\|_1=0$, for every $(i,j)\not=(1,1)$.}
\end{equation}
Let $(i,j)\not=(1,1)$.
By Lemma \ref{lem-commutant}, we have $\lim_{n\rightarrow\mathcal U}\|x_n-P_{1,1}(x_n)\|_2=0$. By using the noncommutative H\"{o}lder inequality and that $1=\frac{1}{2}+3\cdot \frac{1}{6}$,  we get that $$\|x_ny_n^{2,2}x_n^*{y_n^{i,j}}^*-P_{1,1}(x_n)y_n^{2,2}P_{1,1}(x_n)^*{y_n^{i,j}}^*\|_1\leq (\kappa_6^2+\kappa_6^3)\|x_n-P_{1,1}(x_n)\|_2.$$
This implies that \begin{equation}\label{estimate1}
\lim_{n\rightarrow\mathcal U}\|x_ny_n^{2,2}x_n^*{y_n^{i,j}}^*-P_{1,1}(x_n)y_n^{2,2}P_{1,1}(x_n)^*{y_n^{i,j}}^*\|_1=0.\end{equation}
Next, let $v_n\in\mathscr W_{1,1}$ and $w_n^{i,j}\in\mathscr W_{i,j}$ such that $\|P_{1,1}(x_n)-v_n\|_4\leq\frac{1}{n}$ and $\|y_n^{i,j}-w_n^{i,j}\|_4\leq\frac{1}{n}$, for every $n\in\mathbb N$ and $i,j\in\{1,2\}$. Then $\|v_n\|_4,\|w_n^{i,j}\|_4\leq\kappa_4+\frac{1}{n}\leq\kappa_4+1$. Since $1=4\cdot\frac{1}{4}$, applying  the noncommutative H\"{o}lder inequality again gives that
$$\|P_{1,1}(x_n)y_n^{2,2}P_{1,1}(x_n)^*{y_n^{i,j}}^*-v_nw_n^{2,2}v_n^*{w_n^{i,j}}^*\|_1\leq \frac{15(\kappa_4+1)^3}{n}.$$
This implies that 
\begin{equation}
\label{estimate2}
\lim_{n\rightarrow\mathcal U}\|P_{1,1}(x_n)y_n^{2,2}P_{1,1}(x_n)^*{y_n^{i,j}}^*-v_nw_n^{2,2}v_n^*{w_n^{i,j}}^*\|_1=0.
\end{equation}
By combining \eqref{estimate1} and \eqref{estimate2}, it follows that 
\begin{equation}
\label{estimate3}\lim_{n\rightarrow\mathcal U}\|x_ny_n^{2,2}x_n^*{y_n^{i,j}}^*-v_nw_n^{2,2}v_n^*{w_n^{i,j}}^*\|_1=0.
\end{equation}
Now,  $v_nw_n^{2,2}v_n^*{w_n^{i,j}}^*\in\mathscr W_{1,1}\mathscr W_{2,2}\mathscr W_{1,1}\mathscr W_{i,j}^*\subset \mathscr W_{1,1}\mathscr W_{i,j}^*$.
Since $(i,j)\not=(1,1)$, $\mathscr W_{1,1}$ and $\mathscr W_{i,j}$ are orthogonal (algebraic) $B$-bimodules. Thus, $\rE_B(\mathscr W_{1,1}\mathscr W_{i,j}^*)=\{0\}$ and therefore $\rE_B(v_nw_n^{2,2}v_n^*{w_n^{i,j}}^*)=0$, for every $n\in\mathbb N$. In combination with \eqref{estimate3}, this proves \eqref{perp}.

Finally, for every $n\in\mathbb N$, we have that $\|y_n^{2,2}\|_2^2=\langle y_n^{2,2},y_n\rangle=\langle x_ny_n^{2,2},x_ny_n\rangle$. Since $\lim_{n\rightarrow\mathcal U}\|x_ny_n-y_nx_n\|_2=0$, we get that $$\lim_{n\rightarrow\mathcal U}\|y_n^{2,2}\|_2^2=\lim_{n\rightarrow \mathcal U}\langle x_ny_n^{2,2},y_nx_n\rangle=\lim_{n\rightarrow\mathcal U}(\sum_{i,j=1}^2\langle x_ny_n^{2,2},y_n^{i,j}x_n\rangle).$$
On the other hand, \eqref{perp} gives that $\lim_{n\rightarrow\mathcal U}\langle x_ny_n^{2,2},y_n^{i,j}x_n\rangle=0$ if $(i,j)\not=(1,1)$. Thus,  we get $\lim_{n\rightarrow\mathcal U}\|y_n^{2,2}\|_2^2=\lim_{n\rightarrow\mathcal U}\langle x_ny_n^{2,2}, y_n^{1,1}x_n\rangle$, which proves the main assertion. Since $|\langle x_ny_n^{2,2}, y_n^{1,1}x_n\rangle|\leq \|y_n^{2,2}\|_2\|y_n^{1,1}\|_2$, we get that $\lim_{n\rightarrow\mathcal U}\|y_n^{2,2}\|_2\leq\lim_{n\rightarrow\mathcal U}\|y_n^{1,1}\|_2$.
\end{proof}

\begin{lem}\label{3-infinity}
In the setting of Lemma \ref{almost-ortho}, assume additionally  that $\rE_{B^{\mathcal U}}(x^2)=0$ and $y\in\{v\}'\cap M^{\mathcal U}$, for some $v=(v_n)^{\mathcal U}\in\mathscr U(M_2^{\mathcal U})$.
\begin{itemize}
\item [$(\rm i)$] If $\rE_{B^{\mathcal U}}(v)=0$, then  $\lim_{n\rightarrow\mathcal U}\|P_{1,1}(y_n)\|_2=\lim_{n\rightarrow\mathcal U}\|P_{2,2}(y_n)\|_2=0$.
\item [$(\rm ii)$] If $\rE_{B^{\mathcal U}}(v)=\rE_{B^{\mathcal U}}(v^2)=0$, then $y=0$.
\end{itemize}
\end{lem}

\begin{proof}
We keep the notation from the proof of Lemma \ref{almost-ortho}. 

$(\rm i)$ Assume that $\rE_{B^{\mathcal U}}(v)=0$. Write $v=(v_n)^{\mathcal U}$, where $v_n\in M_2\ominus B$, for every $n\in\mathbb N$, and $\sup\left\{\|v_n\|_\infty\mid n\in\mathbb N\right\}<\infty$. 
We first claim that 
\begin{equation}\label{y_11y_22}
\lim_{n\rightarrow\mathcal U}\|x_ny_n^{2,2}-y_n^{1,1}x_n\|_2=0.
\end{equation}
Since $vyv^*=y$, we get  that $\lim_{n\rightarrow\mathcal U}\|v_n^*y_nv_n-y_n\|_2=0$. Thus, we derive that $$\lim_{n\rightarrow\mathcal U}\langle v_ny_n^{1,1}v_n^*,y_n\rangle=\lim_{n\rightarrow\mathcal U}\langle y_n^{1,1},y_n\rangle=\lim_{n\rightarrow\mathcal U}\|y_n^{1,1}\|_2^2.$$ 
Since $(M_2\ominus B)\mathscr W_{1,1}(M_2\ominus B)\subset\mathscr W_{2,2}$ we  also get that $P_{i,j}(v_ny_n^{1,1}v_n^*)=0$, for every $(i,j)\not=(2,2)$ and $n\in\mathbb N$. This implies that $\langle v_ny_n^{1,1}v_n^*,y_n\rangle=\langle v_ny_n^{1,1}v_n^*, y_n^{2,2}\rangle$, for every $n\in\mathbb N$. 
Since $|\langle v_ny_n^{1,1}v_n^*, y_n^{2,2}\rangle\leq \|y_n^{1,1}\|_2\|y_n^{2,2}\|_2$, for every $n\in\mathbb N$, we conclude that $\lim_{n\rightarrow\mathcal U}\|y_n^{1,1}\|_2^2\leq\lim_{n\rightarrow\mathcal U}\|y_n^{1,1}\|_2\|y_n^{2,2}\|_2$ and thus 
\begin{equation}\label{est1}
\lim_{n\rightarrow\mathcal U}\|y_n^{1,1}\|_2\leq\lim_{n\rightarrow\mathcal U}\|y_n^{2,2}\|_2.
\end{equation}
On the other hand, Lemma \ref{almost-ortho} implies that 
\begin{equation}\label{est2}
\lim_{n\rightarrow\mathcal U}\|y_n^{2,2}\|_2^2=\lim_{n\rightarrow\mathcal U}\langle x_ny_n^{2,2},y_n^{1,1}x_n\rangle \quad \text{and} \quad \lim_{n\rightarrow\mathcal U}\|y_n^{2,2}\|_2\leq\lim_{n\rightarrow\mathcal U}\|y_n^{1,1}\|_2.
\end{equation}
It is now clear that \eqref{est1} and \eqref{est2} together imply \eqref{y_11y_22}.
Since $y$ also commutes with $x^*=(x_n^*)$, applying \eqref{y_11y_22} to $x^*$ instead of $x$ gives that $\lim_{n\rightarrow\mathcal U}\|x_n^*y_n^{2,2}-y_n^{1,1}x_n^*\|_2=0$ and thus $\lim_{n\rightarrow\mathcal U}\|y_n^{2,2}x_n-x_ny_n^{1,1}\|_2=0$. In combination with \eqref{y_11y_22}, this implies that 
\begin{equation}\label{est3}
\lim_{n\rightarrow\mathcal U}\|x_n^2y_n^{2,2}-y_n^{2,2}x_n^2\|_2=0.
\end{equation}
Since $y$ commutes with $x^2=(x_n^2)$ and $\rE_{B^{\mathcal U}}(x^2)=0$,  \eqref{perp} from the proof of Lemma \ref{almost-ortho} gives that $\lim_{n\rightarrow\mathcal U}\|\rE_B(x_n^2y_n^{2,2}{x_n^2}^*{y_n^{2,2}}^*)\|_1=0$, thus $\lim_{n\rightarrow\mathcal U}\langle x_n^2y_n^{2,2},y_n^{2,2}x_n^2\rangle=0$. Together with \eqref{est3} we get that $\lim_{n\rightarrow\mathcal U}\|x_n^2y_n^{2,2}\|_2=0$. Since $x_n\in\mathscr U(M)$ we get that $\lim_{n\rightarrow\mathcal U}\|y_n^{2,2}\|_2=0$ and \eqref{est1} gives that $\lim_{n\rightarrow\mathcal U}\|y_n^{1,1}\|_2=0$, proving part $(\rm i)$.

$(\rm ii)$ Assume that $\rE_{B^{\mathcal U}}(v)=\rE_{B^{\mathcal U}}(v^2)=0$. 
By $(\rm i)$, we have $\lim_{n\rightarrow\mathcal U}\|y_n-(y_{1,2}^n+y_{2,1}^n)\|_2=~0$. Since $vy=yv$, we have $\lim_{n\rightarrow\mathcal U}\|v_ny_n-y_n v_n\|_2=0$ and so 
\begin{equation}\label{12_to_21} 
\lim_{n\rightarrow\mathcal U}\|v_ny_{1,2}^n+ v_ny_{2,1}^n - y_{1,2}^n v_n - y_{2,1}^n v_n\|_2=0.
\end{equation}
For every $n \in \N$, we have $v_ny_{1,2}^n = P_{2, 2}(v_ny_{1,2}^n)$, $y_{2,1}^n v_n = P_{2, 2}(y_{2,1}^n v_n)$, $v_ny_{2,1}^n = P_{1,1}(v_ny_{2,1}^n) + P_{2,1}(v_ny_{2,1}^n)$ and $y_{1,2}^n v_n = P_{1, 1}(y_{1,2}^n v_n) + P_{1, 2}(y_{1,2}^n v_n)$. In combination with \eqref{12_to_21}, we obtain 
\begin{equation}\label{12_to_21-upgraded} 
\lim_{n \to \mathcal U} \|v_ny_{2,1}^n - y_{1,2}^n v_n\|_2 = 0 \quad \text{and} \quad \lim_{n \to \mathcal U} \|v_ny_{2,1}^n - P_{1, 1}(v_n y_{2, 1}^n)\|_2 = 0.
\end{equation}
For every $n \in \N$, set $\eta_n = P_{1,1}(v_ny_{2,1}^n) \in \rL^2(M)$. Then \eqref{12_to_21-upgraded}, Theorem \ref{thm-Mei-Ricard} and Lemma \ref{lem-ultraproduct} together imply that $\eta = (\eta_n)^{\mathcal U} = (v_ny_{2,1}^n)^{\mathcal U} = (y_{1,2}^n v_n)^{\mathcal U} \in \rL^2(M^{\mathcal U})$ and that $y = v^* \eta + \eta v^*$. Since $v^* y = y v^*$, we obtain $(v^*)^2 \eta = \eta (v^*)^2$ and so $v^2 \eta = \eta v^2$. Since $\rE_{B^{\mathcal U}}(v^2) = 0$, we may write $v^2 = (w_n)^{\mathcal U}$ where $w_n \in M_2 \ominus B$ for every $n \in \N$. For every $n \in \N$, since $\eta_n = P_{1,1}(\eta_n)$, we have $w_n \eta_n = P_{2, 1}(w_n \eta_n) \perp P_{1, 2}(\eta_n w_n) = \eta_n w_n$. Then we obtain $v^2 \eta \perp \eta v^2$. Since $v^2 \eta = \eta v^2$, this further implies that $v^2\eta = 0$ and so $\eta = 0$. Thus, $y = 0$.
\end{proof}

\begin{proof}[Proof of Theorem \ref{3infinity}] 
Theorem \ref{3infinity} follows directly from part $({\rm ii})$ of Lemma \ref{3-infinity}.
\end{proof}

\section{Proof of Theorem \ref{cor-amenable-absorption}}

\begin{proof}[Proof of Theorem \ref{cor-amenable-absorption}]
Let $P \subset M$ be a von Neumann subalgebra such that $P \cap M_1 \npreceq_{M_1} B$ and $P' \cap M^{\mathcal U} \npreceq_{M^{\mathcal U}} B^{\mathcal U}$. Set $A = P \cap M_1$. By \cite[Theorem 1.1]{IPP05}, since $A \npreceq_{M_1} B$, we have $P' \cap M \subset A' \cap M \subset M_1$ and so $P' \cap M = P' \cap M_1$. The set of projections $p \in P' \cap M_1$ for which $Pp \subset pM_1p$ attains its maximum in a projection $z \in \mathscr Z(P' \cap M_1)$. It suffices to prove that $z = 1$. By contradiction, assume that $z \neq 1$. Set $q = z^\perp$ and $Q = P q$.

\begin{claim}\label{claim-intertwining}
We have $Q \preceq_{M} M_1$.
\end{claim}

\begin{proof}[Proof of Claim \ref{claim-intertwining}]
By contradiction, assume that $Q \npreceq_{M} M_1$. Choose a sequence $(w_k)_k$ in $\mathscr U(Q)$ such that $\lim_k \|\rE_{M_1}(x^* w_k y)\|_2 = 0$ for all $x, y \in qM$. Set $\mathscr Q = Q' \cap (q M q)^{\mathcal U} = q(P' \cap M^{\mathcal U})q$. We have $\mathscr Q \npreceq_{M^{\mathcal U}} B^{\mathcal U}$.

Firstly, we show that $\mathscr Q \preceq_{M^{\mathcal U}} M_1^{\mathcal U}$. By contradiction, assume that $\mathscr Q \npreceq_{M^{\mathcal U}} M_1^{\mathcal U}$. Since $A \npreceq_{M_1} B$, by Lemma \ref{lem-intertwining}, we may choose $u \in \mathscr U(A^{\mathcal U})$ such that $\rE_{B^{\mathcal U}}(a u^m b) = 0$ for all $a, b \in M_1$ and all $m \in \Z \setminus \{0\}$. 
Since $M$ is separable, $\mathscr Q \npreceq_{M^{\mathcal U}} M_1^{\mathcal U}$,  $\mathscr Q\subset A'\cap M$ and $u\in \mathscr U(A^{\mathcal U})$, by a standard diagonal argument, we can construct a unitary $v \in \mathscr U(\mathscr Q)$ such that $\rE_{M_1^{\mathcal U}}(v) = 0$ and $vu=uv$. 
By Lemma \ref{lem-commutant-bis}, the set $\mathbf Y_1 = \{u\}' \cap (M^{\mathcal U} \ominus M_1^{\mathcal U}) $ satisfies $a \mathbf Y_1 b \subset \mathbf X_1$ for all $a, b \in M_1$. On the one hand, applying Theorem \ref{thm-main-result-bis}, since $v \in \mathbf Y_1$,  we have 
$$\forall k \in \N, \quad \rE_{B^{\mathcal U}} \left(v (w_k - \rE_{M_1}(w_k)) v^* (w_k - \rE_{M_1}(w_k))^* \right) = 0.$$
On the other hand, for every $k \in \N$, we have $v w_k = w_k v$ and $\rE_{M_1}(w_k) \to 0$ strongly as $k \to \infty$. Altogether, since $vv^* = v^*v = q = w_kw_k^* = w_k^*w_k$, this implies that $\rE_{B^{\mathcal U}}(q) = 0$, a contradiction. Therefore, we have $\mathscr Q \preceq_{M^{\mathcal U}} M_1^{\mathcal U}$.

Secondly, we derive a contradiction using the proof of \cite[Lemma 9.5]{Io12}. By \cite[Lemma 9.5, Claim 1]{Io12}, there exist $\delta > 0$ and a nonempty finite subset $\mathscr F \subset qM$ such that 
$$\forall v \in \mathscr U(\mathscr Q), \quad \sum_{a, b \in \mathscr F} \|\rE_{M_1^{\mathcal U}}(b^* v a)\|_2^2 \geq \delta.$$
Denote by $\mathbf M_1 \subset M_1^{\mathcal U}$ the set of all elements $x \in M_1^{\mathcal U}$ such that $\rE_{B^{\mathcal U}}(d^* x c) = 0$ for all $c, d \in M_1$. Then denote by $\mathscr K \subset \rL^2((qMq)^{\mathcal U})$ the $\|\cdot\|_2$-closure of the linear span of the set $\left\{a x b^* \mid a, b \in qM, x \in \mathbf M_1\right\}$ and by $e : \rL^2((qMq)^{\mathcal U}) \to \mathscr K$ the corresponding orthogonal projection. 

Since $\mathscr Q \npreceq_{M^{\mathcal U}} B^{\mathcal U}$ and since $M$ is separable, by a standard diagonal argument, we can construct a unitary $v \in \mathscr U(\mathscr Q)$ such that $\rE_{B^{\mathcal U}}(d^* v c) = 0$ for all $c, d \in qM$. Set $\xi = e(v) \in \mathscr K$ and $\eta = \sum_{a, b \in \mathscr F} b \rE_{M_1^{\mathcal U}}(b^* v a) a^* \in (qMq)^{\mathcal U}$. 
Then for every $c,d\in M_1$ and $a,b\in\mathscr F$, we have $\rE_{B^{\mathcal U}}(d^* \rE_{M_1^{\mathcal U}}(b^* v a) c) = \rE_{B^{\mathcal U}}(d^* b^* v a c)=0$.
Thus $\eta \in \mathscr K$ and we have 
$$\langle \xi, \eta \rangle = \langle v, \eta \rangle = \sum_{a, b \in \mathscr F} \|\rE_{M_1^{\mathcal U}}(b^* v a)\|_2^2 \geq \delta.$$
It follows that $\xi = e(v) \neq 0$.

On the one hand, since $\mathscr K \subset \rL^2((qMq)^{\mathcal U})$ is a $qMq$-$qMq$-bimodule and since $v \in \mathscr Q$, for every $k \in \N$, we have $w_k \xi w_k^*= w_k e(v) w_k^* = e(w_k v w_k^*) = e(v) = \xi$. On the other hand, following the proof of \cite[Lemma 9.5, Claim 2]{Io12}, we show that $\lim_k \langle w_k \xi w_k^*, \xi \rangle = 0$. This will give a contradiction. By linearity and density, it suffices to show that $\lim_k \langle w_k \, a_1 x_1 b_1^* \,  w_k^*, a_2 x_2 b_2^*\rangle = 0$ for all $a_1, a_2, b_1, b_2 \in q M$ and all $x_1, x_2 \in \mathbf M_1$. So let us fix $a_1, a_2, b_1, b_2 \in q M$ and  $x_1, x_2 \in \mathbf M_1$. We may further assume that $\max \left\{\|a_i\|_\infty, \|b_i\|_\infty, \|x_i\|_\infty \mid i \in \{1, 2\}\right\} \leq 1$. Then for every $k \in \N$, we have
$$|\langle w_k \, a_1 x_1 b_1^* \,  w_k^*, a_2 x_2 b_2^*\rangle| = |\tau^{\mathcal U}(x_2^* a_2^* w_k a_1 x_1 b_1^*  w_k^* b_2)| \leq \|\rE_{M_1^{\mathcal U}}( a_2^* w_k a_1 \, x_1 \, b_1^*  w_k^* b_2)\|_2.$$
Using the amalgamated free product structure $M = M_1 \ast_B M_2$, the inclusion $M_1 \subset M$ is mixing relative to $B$. In particular, since $x_1 \in \mathbf M_1$, we have $\rE_{M_1^{\mathcal U}}(c^* x_1 d) = \rE_{M_1^{\mathcal U}}(c^* x_1)= \rE_{M_1^{\mathcal U}}(x_1 d)=0$ for all $c, d \in M \ominus M_1$ (see e.g.\! the proof of \cite[Claim 2.5]{CH08}). This implies 
$$\forall k \in \N, \quad \rE_{M_1^{\mathcal U}}( a_2^* w_k a_1 \, x_1 \, b_1^* w_k^* b_2) = \rE_{M_1}(a_2^* w_k a_1) \, x_1 \rE_{M_1}(b_1^* w_k^* b_2).$$ 
Thus, we have 
$$\limsup_k |\langle w_k \, a_1 x_1 b_1^* \,  w_k^*, a_2 x_2 b_2^* \rangle| \leq \limsup_k \| \rE_{M_1}(a_2^* w_k a_1)\|_2 = 0.$$
This gives a contradiction and finishes the proof of Claim \ref{claim-intertwining}.
\end{proof}

Since $Q \preceq_{M} M_1$, there exist $n \geq 1$, a projection $r \in \mathbf M_n(M_1)$, a nonzero partial isometry $v = [v_1, \dots, v_n] \in \mathbf M_{1, n}(z^\perp M)r$ and a unital normal $\ast$-homomorphism $\pi : Q \to r\mathbf M_n(M_1)r$  such that $a v = v \pi(a)$ for all $a \in Q$. In particular, we have $A v_i \subset \sum_{j = 1}^n v_j M_1$ for every $i \in \{1, \dots, n\}$. By \cite[Theorem 1.1]{IPP05}, since $A \npreceq_{M_1} B$, we have $v_i \in M_1$ for every $i \in \{1, \dots, n\}$. It follows that $vv^* \in Q' \cap M_1$ and  $Q vv^* \subset vv^* M_1 vv^*$. Thus, we obtain $P(z + vv^*) \subset (z + vv^*) M_1 (z + vv^*)$. This contradicts the maximality of the projection $z \in P' \cap M_1$. Therefore, we have $z = 1$ and so $P \subset M_1$.
\end{proof}

\begin{rem}\label{rem-amenable}
We make two observations.

\begin{itemize}

\item [$(\rm i)$] If $A \subset M_1$ is a von Neumann subalgebra such that $A \npreceq_{M_1} B$, then we have $A \npreceq_M B$. Indeed, this follows from the amalgamated free product structure $M = M_1 \ast_B M_2$ and the fact that the inclusion $M_1 \subset M$ is mixing relative to $B$ (see the proof of Claim \ref{claim-intertwining}).

\item [$(\rm ii)$] If $P \subset M$ is an amenable von Neumann subalgebra such that $P \npreceq_{M} B$, then we have $P' \cap M^{\mathcal U} \npreceq_{M^{\mathcal U}} B^{\mathcal U}$. Indeed, by contradiction, assume that $P' \cap M^{\mathcal U} \preceq_{M^{\mathcal U}} B^{\mathcal U}$. On the one hand, by \cite[Lemma 9.5, Claim 1]{Io12}, there exist $\delta > 0$ and a nonempty finite subset $\mathscr F \subset M$ such that 
\begin{equation}\label{eq-embedding}
\forall v \in \mathscr U(P' \cap M^{\mathcal U}), \quad \sum_{a, b \in \mathscr F} \|\rE_{B^{\mathcal U}}(b^* v a)\|_2^2 \geq \delta.
\end{equation}
On the other hand, since $P$ is amenable hence hyperfinite by Connes' fundamental result \cite{Co75}, there exists an increasing sequence $(P_k)_k$ of finite dimensional von Neumann subalgebras of $P$ such that $(\bigcup_kP_k)\dpr=P$ and $P_k' \cap P \subset P$ has finite index for every $k \in \N$ (see e.g.\! the proof of \cite[Theorem 8.1]{Ho12}). Since $P \npreceq_{M} B$, it follows that $P_k'\cap P \npreceq_{M} B$ for every $k \in \N$. Since $M$ is separable, by a standard diagonal argument, we can construct a unitary $v \in \mathscr U(P' \cap M^{\mathcal U})$ such that $\rE_{B^{\mathcal U}}(b^* v a) = 0$ for all $a, b \in M$. This contradicts \eqref{eq-embedding}. Therefore, we have $P' \cap M^{\mathcal U} \npreceq_{M^{\mathcal U}} B^{\mathcal U}$. 
\end{itemize}
\end{rem}


\section{A lifting theorem and proofs of Theorems \ref{inductive} and  \ref{orthogonal}}

\subsection{A lifting theorem}
The goal of this subsection is to establish the following lifting theorem which will be needed in the proof of Theorem \ref{inductive}.
\begin{thm}\label{lifting}
Let $\mathcal U$ be an ultrafilter on a set $K$ and $(M_k,\tau_k),k\in K$, be tracial von Neumann algebras. Let $A,B\subset\prod_{\mathcal U}M_k$ be separable abelian von Neumann subalgebras which are $2$-independent in $\prod_{\mathcal U}M_k$ with respect to $(\tau_k)^{\mathcal U}$.  Then there exist orthogonal abelian von Neumann subalgebras $C_k,D_k\subset M_k$, for every $k\in K$, such that $A\subset\prod_{\mathcal U}C_k$ and $B\subset\prod_{\mathcal U}D_k$.
\end{thm}
We do not know whether Theorem \ref{lifting} still holds if we replace the assumption that $A$ and $B$ are $2$-independent with the weaker assumption that $A$ and $B$ are orthogonal. When $\dim(A)=2$ and $\dim(B)=3$, Theorem \ref{lifting} follows from \cite[Lemma 3.1]{CIKE22}, which moreover only assumes that $A$ and $B$ are orthogonal. Theorem \ref{lifting} is new in all other cases, including when $A$ and $B$ are finite dimensional and of dimension at least $3$.

The proof of Theorem \ref{lifting} relies on the following perturbation lemma. First, we need to introduce some additional terminology.
Let $(M,\tau)$ be a tracial von Neumann algebra. We denote by $M_{\text {sa},1}$ the set of $x\in M$ such that $x=x^*$ and $\|x\|_\infty \leq 1$. 
Let $x=(x_1,\dots, x_m)\in M^m$ and $y=(y_1,\dots,y_n)\in M^n$, for some $m,n\in\mathbb N$. For $u\in\mathscr U(M)$, we write $uxu^*=(ux_1u^*, \dots, ux_mu^*)$. We define 
\begin{align*}
\delta(x,y) &=\min \left\{\|[x_i,y_j]\|_2\mid 1\leq i\leq m,1\leq j\leq n \right\}, \\
\varepsilon(x,y) &=\max \left\{|\tau(x_iy_j)|\mid 1\leq i\leq m,1\leq j\leq n \right\}, \\
\gamma(x,y) &=\max \left\{|\langle [x_i,y_j],[x_{i'},y_{j'}]\rangle|\mid 1\leq i,i'\leq m,1\leq j,j'\leq n, (i,j)\not=(i',j') \right\}.
\end{align*}

\begin{lem}\label{perturbation}
Let $(M,\tau)$ be a tracial von Neumann algebra, $x=(x_1,\dots, x_m)\in M_{\emph{sa},1}^m$ and $y=(y_1,\dots,y_n)\in M_{\emph{sa},1}^n$, for  $m,n\in\mathbb N$. Set $\delta_0=\delta(x,y), \varepsilon_0=\varepsilon(x,y),\gamma_0=\gamma(x,y)$. Assume that  $13mn\sqrt{\varepsilon_0} < \delta_0^2 - (mn-1)\gamma_0$. Then there exists $v\in\mathscr U(M)$ such that  
$$\|v-1\|_\infty \leq \frac{8mn\varepsilon_0}{\delta_0^2-(mn-1)\gamma_0} \leq\frac{8}{13}\sqrt{\varepsilon_0} \quad \text{and} \quad  \varepsilon(vxv^*,y)=0.$$
\end{lem}

Note that Lemma \ref{perturbation} is interesting even when $M$ is finite dimensional.
To prove Lemma \ref{perturbation}, we will need two auxiliary lemmas. 

\begin{lem}\label{linearindep}
Let $(M,\tau)$ be a tracial von Neumann algebra, $\xi_1,\dots,\xi_p\in M_{\emph{sa},1}$ and $\alpha_1,\dots,\alpha_p\in\mathbb R$, for some $p\geq 2$. Let $\delta\in (0,1)$ and $\varepsilon\in (0,\frac{\delta^2}{p-1})$. Assume that $\|\xi_i\|_2 \geq\delta$, for every $1\leq i\leq p$, and $|\langle \xi_i,\xi_j\rangle|\leq\varepsilon$, for every $1\leq i<j\leq p$.
Then there exists $h\in M$ such that $h=h^*, \|h\|_\infty\leq\frac{\sum_{j=1}^p|\alpha_j|}{\delta^2 -(p-1)\varepsilon}$ and $\tau(h\xi_i)=\alpha_i$, for every $1\leq i\leq p$.
\end{lem}

\begin{proof}
First, we claim that $\xi_1,\dots,\xi_p$ are linearly independent. Otherwise, we can find $\beta_1,\dots,\beta_p\in\mathbb R$ such that $\beta_1\xi_1+\cdots+\beta_p\xi_p=0$ and $\max \left\{|\beta_i|\mid 1\leq i\leq p \right\}>0$. Let $1\leq j\leq p$ such that $|\beta_j|=\max \left\{|\beta_i|\mid 1\leq i\leq p \right\}$. Then $-\beta_j\xi_j=\sum_{i\not=j}\beta_i\xi_i$ and thus $|\beta_j|\, \|\xi_j\|_2^2 \leq\sum_{i\not=j}|\beta_i|\, |\langle\xi_i,\xi_j\rangle|\leq |\beta_j|\sum_{i\not=j}|\langle\xi_i,\xi_j\rangle|$. Since $\beta_j\not=0$, we derive that $\|\xi_j\|_2^2\leq\sum_{i\not=j}|\langle \xi_i,\xi_j\rangle|$, which implies that $\delta^2\leq (p-1)\varepsilon$, contradicting that $\delta^2 > (p-1)\varepsilon$. 

Since $\xi_1,\dots,\xi_p$ are linearly independent, it follows that we can find $\lambda_1,\dots,\lambda_p\in\mathbb R$ such that $h=\sum_{i=1}^p\lambda_i\xi_i$ satisfies $\tau(h\xi_j)=\langle h,\xi_j\rangle=\alpha_j$, for every $1\leq j\leq p$. Then $|\alpha_j|=|\sum_{i=1}^p\lambda_i\langle\xi_i,\xi_j\rangle|\geq |\lambda_j|\|\xi_j\|_2^2-\sum_{i\not=j}|\lambda_i||\langle\xi_i,\xi_j\rangle|$ and thus
\begin{equation}\label{lambda}
\forall 1\leq j\leq p, \quad  |\alpha_j|\geq \delta^2 |\lambda_j|-\varepsilon\sum_{i\not=j}|\lambda_i|.
\end{equation}
Adding the inequalities in \eqref{lambda} for $1\leq j\leq p$ gives $\sum_{j=1}^p|\alpha_j|\geq (\delta^2 - (p-1)\varepsilon)\sum_{j=1}^p|\lambda_j|$. Thus, $\|h\|_\infty \leq\sum_{j=1}^p|\lambda_j|\leq \frac{\sum_{j=1}^p|\alpha_j|}{\delta^2-(p-1)\varepsilon}$. Since $h=h^*$, this finishes the proof.
\end{proof}

\begin{lem}\label{moveone}
Let $(M,\tau)$ be a tracial von Neumann algebra, $x=(x_1,\dots, x_m)\in M_{\emph{sa},1}^m$ and $y=(y_1,\dots,y_n)\in M_{\emph{sa},1}^n$, for some $m,n\in\mathbb N$. 
Set $\delta=\delta(x,y), \varepsilon=\varepsilon(x,y),\gamma=\gamma(x,y)$. Assume that $2mn\varepsilon<\delta^2-(mn-1)\gamma$ and set $\lambda=\frac{2 mn\varepsilon}{\delta^2 -(mn-1)\gamma}<1$.

Then there exists $u\in\mathscr U(M)$ such that 
\begin{itemize}
\item [$(\rm i)$] $\|u-1\|_\infty \leq 2\lambda$.
\item [$(\rm ii)$] $\delta(uxu^*,y)\geq \delta-8\lambda$.
\item [$(\rm iii)$] $\varepsilon(uxu^*,y)\leq 4\lambda^2$.
\item [$(\rm iv)$] $\gamma(uxu^*,y)\leq\gamma+32\lambda$.
\end{itemize}
\end{lem}

\begin{proof}
For every $1\leq i\leq m,1\leq j\leq n$, set $\xi_{i,j}=-\frac{{\rm i}}{2}[x_i,y_j]$. Then $\xi_{i,j}\in M_{\text{sa},1}$ and $\|\xi_{i,j}\|_2=\frac{\|[x_i,y_j]\|_2}{2}\geq \frac{\delta}{2}$, for every $1\leq i\leq m,1\leq j\leq n$. On the other hand, for every $(i,j)\not=(i',j')$, we have $|\langle\xi_{i,j},\xi_{i',j'}\rangle|=\frac{|\langle [x_i,y_j],[x_{i'},y_{j'}]\rangle|}{4}\leq\frac{\gamma}{4}$. 
 
 By applying Lemma \ref{linearindep} to $\xi_{i,j}$ and $\alpha_{i,j}=\frac{\tau(x_iy_j)}{2}$, we may find $h\in M$ such that $h=h^*$, 
 \begin{equation}\label{h} 
 \forall 1\leq i\leq m,1\leq j\leq m, \quad \tau(h\xi_{i,j})=\frac{\tau(x_iy_j)}{2}, 
 \end{equation}
 and
 \begin{equation}\label{h2}
 \|h\|_\infty\leq \frac{\sum_{i,j}\frac{|\tau(x_iy_j)|}{2}}{\frac{\delta^2}{4}-(mn-1)\frac{\gamma}{4}}\leq\frac{2mn\varepsilon}{\delta^2 - (mn-1)\gamma}=\lambda.
\end{equation}

Define $u=\exp({\rm i} h)\in\mathscr U(M)$. We will prove that $u$ satisfies the conclusion. Since for every $x\in\mathbb R$, $|\exp({\rm i} x)-1|\leq 2|x|$ and $|\exp({\rm i} x)-(1+ {\rm i} x)|\leq x^2$,  using \eqref{h2} we get that
\begin{equation}\label{u}
\|u-1\|_\infty\leq 2\lambda \quad \text{and} \quad \|u-(1+ {\rm i} h)\|_\infty \leq \lambda^2.
\end{equation}

Let $1\leq i\leq m$ and $1\leq j\leq n$. Then using  \eqref{h2} and the second part of \eqref{u} we get that $\|ux_iu^*y_j-(1+ {\rm i} h)x_i(1+ {\rm i} h)^*y_j\|_\infty \leq \|u-(1+ {\rm i} h)\|_\infty (1+\|1+ {\rm i} h\|_\infty)\leq \lambda^2(2+\lambda)\leq 3\lambda^2$ and $\|(1+ {\rm i} h)x_i(1+ {\rm i} h)^*y_j-(x_iy_j+ {\rm i} (hx_iy_j-x_ihy_j))\|_\infty =\|hx_ihy_j\|_\infty \leq \lambda^2.$
Thus, we get 
$$\|ux_iu^*y_j-(x_iy_j+{\rm i} (hx_iy_j-x_ihy_j))\|_\infty \leq 4\lambda^2$$ 
and therefore
$|\tau(ux_iu^*y_j)-\tau(x_iy_j+ {\rm i} (hx_iy_j-x_ihy_j))|\leq 4\lambda^2$. On the other hand, \eqref{h} gives $\tau(x_iy_j+ {\rm i} (hx_iy_j-x_ihy_j))=\tau(x_iy_j)+\tau({\rm i} h[x_i,y_j])=\tau(x_iy_j)-2\tau(h\xi_{i,j})=0$. Altogether, we get that  $|\tau(ux_iu^*y_j)|\leq 4\lambda^2$. 
Thus, $\varepsilon(uxu^*,y)\leq 4\lambda^2$, which proves $(\rm iii)$.

Next, $\|[ux_iu^*,y_j]-[x_i,y_j]\|_2\leq 2\|ux_iu^*-x_i\|_2\leq 4\|u-1\|_2\leq 8\lambda$, by the first part of \eqref{u}. Hence, $\|[ux_iu^*,y_j]\|_2\geq \|[x_i,y_j]\|_2-8\lambda\geq\delta-8\lambda$, for every $1\leq i\leq m$ and $1\leq j\leq n$.
This implies that $\delta(uxu^*,y)\geq\delta-8\lambda$, which proves $(\rm ii)$.

Finally, 
 for every $(i,j),(i',j')$ we have  $\|[ux_{i'}u^*,y_{j'}]\|_2\leq 2$, $\|[x_i,y_j]\|_2\leq 2$ and thus 
 \begin{align*}
 &|\langle [ux_iu^*,y_j],[ux_{i'}u^*,y_{j'}]\rangle-\langle [x_i,y_j],[x_{i'},y_{j'}]\rangle|\\
 & \quad \leq 2\big(\|[ux_iu^*,y_j]-[x_i,y_j]\|_2+\|[ux_{i'}u^*,y_{j'}]-[x_{i'},y_{j'}]\|_2\big)\leq 32\lambda. 
 \end{align*}
Thus, $|\langle [ux_iu^*,y_j],[ux_{i'}u^*,y_{j'}]\rangle|\leq |\langle [x_i,y_j],[x_{i'},y_{j'}]\rangle|+32\lambda\leq \gamma+32\lambda$. This implies that $\gamma(uxu^*,y)\leq\gamma+32\lambda$, which proves $(\rm iv)$. Since $(\rm i)$ also holds by the first part of \eqref{u}, this finishes the proof.
\end{proof}

\begin{proof}[Proof of Lemma \ref{perturbation}]
We will inductively construct sequences $(u_k)_{k\in\mathbb N}\subset\mathscr U(M)$ and $(\lambda_k)_{k\in\mathbb N}\subset (0,\infty)$ with the following properties: $\lambda_0=1$, $\lambda_1=\frac{2mn\varepsilon_0}{\delta_0^2-(mn-1)\gamma_0}$
and if we define
 $v_0=1$, $v_k=u_ku_{k-1}\cdots u_1\in\mathscr U(M)$, $\delta_k=\delta(v_kxv_k^*,y),\varepsilon_k=\varepsilon(v_kxv_k^*,y)$ and $\gamma_k=\gamma(v_kxv_k^*,y)$, for every $k\geq 0$, then for every $k\geq 1$ we have that 
\begin{enumerate}[(i)]
\item $\|u_k-1\|_\infty \leq 2\lambda_k$.
\item $\delta_k\geq \delta_{k-1}-8\lambda_k$.
\item $\varepsilon_k\leq 4\lambda_k^2$.
\item $\gamma_k\leq\gamma_{k-1}+32\lambda_k$.
\item $\lambda_k\leq\frac{\lambda_{k-1}}{2}.$
\end{enumerate}

Since $\varepsilon_0\leq 1$, we have
that $4mn\varepsilon_0\leq 13mn\sqrt{\varepsilon_0} < \delta_0^2 - (mn-1)\gamma_0$.  Thus, $\lambda_1<\frac{1}{2}$ and hence condition (v) holds for $k=1$.
By applying Lemma \ref{moveone}, we can find $u_1\in\mathscr U(M)$ such that conditions (i)-(iv) hold for $k=1$. 

Next, assume that we have constructed $u_1,\dots, u_l\in\mathscr U(M)$ and $\lambda_1,\dots,\lambda_l\in (0,\infty)$, for some $l\in\mathbb N$, such that conditions (i)-(v) are satisfied for $k=1,\dots,l$. 
Our goal is to construct $u_{l+1}$ and $\lambda_{l+1}$. 
Let $\lambda_{l+1}=\frac{2mn\varepsilon_l}{\delta_l^2-(mn-1)\gamma_l}$. We continue with the following claim.

\begin{claim}\label{lambda0}$\lambda_{l+1}\leq\frac{\lambda_l}{2}$. \end{claim}

\noindent
{\it Proof of Claim \ref{lambda0}.}
First, $(\rm ii)$ implies that $\delta_k^2 \geq (\delta_{k-1}-8\lambda_k)^2 \geq \delta_{k - 1}^2 - 32 \lambda_k$. Then combining (ii) and (iv) gives that 
$$\forall 1\leq k\leq l, \quad \delta_k^2-(mn-1)\gamma_k\geq (\delta_{k-1}^2-(mn-1)\gamma_{k - 1}) - 32mn\lambda_k$$ 
which implies that $\delta_l^2-(mn-1)\gamma_l\geq (\delta_0^2-(mn-1)\gamma_0)-  32mn(\sum_{k=1}^l\lambda_k).$ By using that (v) holds for $k=1,\dots,l$, we also get that $\sum_{k=1}^l\lambda_k\leq 2\lambda_1$. By combining the last two inequalities we get that 
\begin{equation}\label{lambda1}
\delta_l^2-(mn-1)\gamma_l\geq (\delta_0^2-(mn-1)\gamma_0) - 64mn\lambda_1.
\end{equation}
Since $13mn\sqrt{\varepsilon_0} < \delta_0^2 - (mn-1)\gamma_0$, we get that $(\delta_0^2-(mn-1)\gamma_0)^2 >  169(mn)^2 \varepsilon_0$ and thus
\begin{equation}\label{lambda2}
\delta_0^2-(mn-1)\gamma_0 > 80mn\lambda_1.
\end{equation}
By combining \eqref{lambda1} and \eqref{lambda2} we derive that
\begin{equation}\label{lambda3}
\delta_l^2-(mn-1)\gamma_l\geq 16mn\lambda_1.
\end{equation}
Since (v) holds for every $k=1,\dots,l$, we get that $\lambda_l\leq\lambda_1$. Since $\varepsilon_l\leq 4\lambda_l^2$ by (iii), using \eqref{lambda3} we get that
$$\lambda_{l+1}=\frac{2 mn\varepsilon_l}{\delta_l^2 -(mn-1)\gamma_l}\leq \frac{8mn\lambda_l^2}{\delta_l^2-(mn-1)\gamma_l}\leq\frac{16mn\lambda_1}{\delta_l^2-(mn-1)\gamma_l}\cdot\frac{\lambda_l}{2}\leq\frac{\lambda_l}{2}.$$
This finishes the proof of the claim.
\noindent\hfill$\square$

By using (v) and Claim \ref{lambda0} we get that $\lambda_{l+1}\leq\frac{1}{2^{l+1}}<1$. Thus, $2 mn\varepsilon_l < \delta_l^2 -(mn-1)\gamma_l$. We can therefore apply Lemma \ref{moveone} to $v_lxv_l^*$ and $y$ to find $u_{l+1}\in\mathscr U(M)$ such that
\begin{enumerate}[(i')]
\item $\|u_{l+1}-1\|_\infty\leq 2\lambda_{l+1}$.
\item $\delta_{l+1}=\delta(u_{l+1}(v_lxv_l^*)u_{l+1}^*),y)\geq \delta_{l}-8\lambda_{l+1}$.
\item $\varepsilon_{l+1}=\varepsilon(u_{l+1}(v_lxv_l^*)u_{l+1}^*,y)\leq 4\lambda_{l+1}^2$.
\item $\gamma_{l+1}=\gamma(u_{l+1}(v_lxv_l^*)u_{l+1}^*,y)\leq\gamma_l+32\lambda_{l+1}$.
\end{enumerate}
By induction, this finishes the construction of $(u_k)_{k\in\mathbb N}\subset\mathscr U(M)$ and $(\lambda_k)_{k\in\mathbb N}\subset (0,\infty)$.

Finally, since $\lambda_0=1$, (v) implies that $\lambda_k\leq \frac{1}{2^k}$, for every $k\geq 0$. Using (i), we derive that $\|v_k-v_{k-1}\|_\infty=\|u_k-1\|_\infty\leq\frac{1}{2^{k-1}}$, for every $k\geq 1$. Thus, the sequence $(v_k)_{k\in\mathbb N}$ is Cauchy in $\|\cdot\|_\infty$  and so we can find $v\in\mathscr U(M)$ such that $\lim_{k\rightarrow\infty}\|v_k-v\|_\infty=0$. 
Using (iii), we get that $\varepsilon_k\leq 4\lambda_k^2\leq \frac{1}{4^{k-1}}$, for every $k\geq 1$. Thus, $\varepsilon(vxv^*,y)=\lim_{k\rightarrow\infty}\varepsilon_k=0$.
Moreover, using (i) and (v) we get that $\|v_k-1\|_\infty \leq\sum_{l=1}^k\|u_l-1\|_\infty\leq\sum_{l=1}^k2\lambda_l\leq 4\lambda_1.$ Hence $\|v-1\|_\infty=\lim_{k\rightarrow\infty}\|v_k-1\|_\infty \leq 4\lambda_1= \frac{8mn\varepsilon_0}{\delta_0^2-(mn-1)\gamma_0}$. This finishes the proof.
\end{proof}

 \begin{proof}[Proof of Theorem \ref{lifting}] We may clearly assume that $\dim(A)\geq 2$ and $\dim(B)\geq 2$.
Since $A$ and $B$ are separable, we can write $A=(\bigcup_{n\in\mathbb N}A_n)\dpr$, $B=(\bigcup_{n\in\mathbb N}B_n)\dpr$, where $A_n\subset A, B_n\subset B$ are finite dimensional von Neumann subalgebras such that  $A_n\subset A_{n+1}$,  $B_n\subset B_{n+1}$, $a_n:=\dim(A_n)\geq 2$ and $b_n:=\dim(B_n)\geq 2$, for every $n\in\mathbb N$. 

Fix $n\in\mathbb N$. Write $A_n=\bigoplus_{i=1}^{a_n}\mathbb Cp_{n,i}$ and $B_n=\bigoplus_{j=1}^{b_n}\mathbb Cq_{n,j}$, where $(p_{n,i})_{i=1}^{a_n}$ and $(q_{n,j})_{j=1}^{b_n}$ are partitions of unity into projections from $A$ and $B$, respectively. For every $1\leq i\leq a_n$ and $1\leq j\leq b_n$, represent $p_{n,i},q_{n,j}\in\prod_{\mathcal U}M_k$ as $p_{n,i}=(p_{n,i}^k)^{\mathcal U}$ and $q_{n,j}=(q_{n,j}^k)^{\mathcal U}$, where for every $k\in K$, $(p_{n,i}^k)_{i=1}^{a_n}$ and $(q_{n,j}^k)_{j=1}^{b_n}$ are partitions of unity into projections from $M_k$. Denote $A_n^k=\bigoplus_{i=1}^{a_n}\mathbb Cp_{n,i}^k$ and $B_n^k=\bigoplus_{j=1}^{b_n}\mathbb Cq_{n,j}^k$.
Moreover, we can arrange that $A_n^k\subset A_{n+1}^k$ and $B_n^k\subset B_{n+1}^k$, for every $n\in\mathbb N$ and $k\in K$.

If $(r_l)_{l=1}^m$ is a partition of unity into nonzero projections from a tracial von Neumann algebra $(N,\tau)$, then $\left\{\tau(r_{l+1}+\cdots+r_m)r_l-\tau(r_l)(r_{l+1}+\cdots+r_m)\mid 1\leq l\leq m-1 \right\}$ is an orthogonal basis for $C\ominus\mathbb C1$ contained in $C_{\text{sa},1}$, where $C=\bigoplus_{l=1}^m\mathbb Cr_l$.
Using this observation,  for every $1\leq i\leq a_n-1,1\leq j\leq b_n-1$ and $k\in K$, we define $$x_{n,i}=\tau(p_{n,i+1}+\cdots+p_{n,a_n})p_{n,i}-\tau(p_{n,i})(p_{n,i+1}+\cdots+p_{n,a_n}),$$ $$y_{n,j}=\tau(q_{n,j+1}+\cdots+q_{n,b_n})q_{n,j}-\tau(q_{n,j})(q_{n,j+1}+\cdots+q_{n,b_n}),$$ 
$$x_{n,i}^k=\tau(p_{n,i+1}^k+\cdots+p_{n,a_n}^k)p_{n,i}^k-\tau(p_{n,i}^k)(p_{n,i+1}^k+\cdots+p_{n,a_n}^k),$$ $$y_{n,j}^k=\tau(q_{n,j+1}^k+\cdots+q_{n,b_n}^k)q_{n,j}^k-\tau(q_{n,j}^k)(q_{n,j+1}^k+\cdots+q_{n,b_n}^k).$$ 

Set $x_n=(x_{n,i})_{i=1}^{a_n-1}\in A_n^{a_n-1}, y_n=(y_{n,j})_{j=1}^{b_n-1}\in B_n^{b_n-1}, x_n^k=(x_{n,i}^k)_{i=1}^{a_n-1}\in M_k^{a_n-1}$ and  $y_n^k=(y_{n,j}^k)_{j=1}^{b_n-1}\in M_k^{b_n-1}$.
Let $n\in\mathbb N$, $1\leq i,i'\leq a_n-1$ and $1\leq j,j'\leq b_n-1$ with $(i,j)\not=(i',j')$. Since $A_n$ and $B_n$ are $2$-independent, $x_{n,i}\not=0$ and $y_{n,j}\not=0$, we have that $\|[x_{n,i},y_{n,j}]\|_2= \sqrt{2}\|x_{n,i}\|_2\|y_{n,j}\|_2>0$ and $\tau(x_{n,i}y_{j,n})=0$. Moreover,   $\langle [x_{n,i},y_{n,j}],[x_{n,i'},y_{n,j'}]\rangle=2\tau(x_{n,i}x_{n,i'})\tau(y_{n,j}y_{n,j'}).$ 
Since $(x_{n,i})_{i=1}^{a_n-1}$ and $(y_{n,j})_{j=1}^{b_n-1}$ are pairwise orthogonal, we get that $\langle [x_{n,i},y_{n,j}],[x_{n,i'},y_{n,j'}]\rangle=0$.
Altogether, we derive that $\delta(x_n,y_n)>0$ and $\varepsilon(x_n,y_n)=\gamma(x_n,y_n)=0$. 

Thus, we get that $\lim_{k\rightarrow\mathcal U}\delta(x_n^k,y_n^k)=\delta(x_n,y_n)>0$, $\lim_{k\rightarrow\mathcal U}\varepsilon(x_n^k,y_n^k)=\varepsilon(x_n,y_n)=0$ and $\lim_{k\rightarrow\mathcal U}\gamma(x_n^k,y_n^k)=\gamma(x_n,y_n)=0$. By applying Lemma \ref{perturbation}, we find $v_n^k\in\mathscr U(M_k)$, for every $k\in K$, such that $\varepsilon(v_n^kx_n^k{v_n^k}^*,y_n^k)=0$, for every $k\in K$, and $\lim_{k\rightarrow\mathcal U}\|v_n^k-1\|_\infty=~0$. Since $x_n^k$ and $y_n^k$ are bases for $A_n^k$ and $B_n^k$, respectively, we get that $v_n^kA_n^k{v_n^k}^*$ and $B_n^k$ are orthogonal, for every $k\in K$.

To complete the proof we consider two cases:

\vskip 0.1in
{\bf Case 1.} $\mathcal U$ is countably cofinal. 

In this case, we proceed as in the proof of \cite[Lemma 2.2]{BCI15}. Since $\mathcal U$ is countably cofinal, there exists a decreasing sequence $\{S_n\}_{n\geq 2}$ of sets in $\mathcal U$ such that $\bigcap_{n\geq 2}S_n=\emptyset$. 
For $n\geq 2$, let $T_n=\{k\in K\mid \|v_m^k-1\|_\infty<\frac{1}{n}, \forall 1\leq m\leq n\}\in\mathcal U$ and set $K_n=S_n\cap T_n$. Then $\{K_n\}_{n\geq 2}$ is a decreasing sequence of sets in $\mathcal U$ such that $\bigcap_{n\geq 2}K_n=\emptyset$. Let $K_1=K\setminus K_2$. For every $k\in K$, let $n(k)$ be the smallest integer $n\geq 1$ such that $k\in K_n$. Then $n(k)$ is well-defined and $\lim_{k\rightarrow\mathcal U}n(k)=+\infty$. 

For $k\in K$, let $C_k=A_{n(k)}^k, D_k^0=B_{n(k)}^k$ and $v_k=v_{n(k)}^k$. If $n(k)\geq 2$, then as $k\in K_{n(k)}$ we have $\|v_k-1\|_\infty < \frac{1}{n(k)}$.  Since $\lim_{k\rightarrow\mathcal U}n(k)=+\infty$, we get that $\lim_{k\rightarrow\mathcal U}\|v_k-1\|_\infty =0$. 

Let $n\in\mathbb N$. Since $\{k\in K\mid n(k)\geq n\}\in\mathcal U$ and the sequences  $\{A_m^k\}_{m\in\mathbb N}$ and $\{B_m^k\}_{m\in\mathbb N}$ are increasing for every $k\in K$, we get that $\prod_{\mathcal U}A_n^k\subset\prod_{\mathcal U}C_k$ and $\prod_{\mathcal U}B_n^k\subset\prod_{\mathcal U}D_k^0$. Since $A_n\subset\prod_{\mathcal U}A_n^k$ and $B_n\subset\prod_{\mathcal U}B_n^k$, we conclude that $A_n\subset\prod_{\mathcal U}C_k$ and $B_n\subset\prod_{\mathcal U}D_k^0$. As this holds for every $n\in\mathbb N$, we get that $A\subset\prod_{\mathcal U}C_k$ and $B\subset\prod_{\mathcal U}D_k^0$. Finally, let $D_k=v_kD_k^0v_k^*$. Then $C_k=A_{n(k)}^k$ and $D_k=v_{n(k)}^kB_{n(k)}^kv_{n(k)}^*$ are orthogonal, for every $k\in K$. Since $\lim_{k\rightarrow\mathcal U}\|v_k-1\|_\infty=0$, we get that $\prod_{\mathcal U}D_k^0=\prod_{\mathcal U}D_k$ and $B\subset\prod_{\mathcal U}D_k$.
This finishes the proof of Case 1.
\vskip 0.1in
{\bf Case 2.} $\mathcal U$ is not countably cofinal.

Since $\mathcal U$ is not countably cofinal, $\{k'\in K\mid f(k')=\lim_{k\rightarrow\mathcal U}f(k)\}\in\mathcal U$, for every $f\in\ell^\infty(K)$ (see the proof of \cite[Lemma 2.3 (2)]{BCI15}). If $n\in\mathbb N$, since $\lim_{k\rightarrow\mathcal U}\|v_n^k-1\|_\infty=0$, we get that $R_n:=\{k\in K\mid v_n^k=1\}\in\mathcal U$. Using again that $\mathcal U$ is not countably cofinal, we further deduce that $R:=\bigcap_{n\in\mathbb N}R_n=\{k\in K\mid v_n^k=1,\forall n\in\mathbb N\}\in\mathcal U$. 

If $k\in R$, then $v_n^k=1$, hence $A_n^k$ and $B_n^k$ are orthogonal, for every $n\in\mathbb N$. Since the sequences  $\{A_n^k\}_{n\in\mathbb N}$ and $\{B_n^k\}_{n\in\mathbb N}$ are increasing, we get that $C_k=(\bigcup_{n\in\mathbb N}A_n^k)\dpr$ and $D_k=(\bigcup_{n\in\mathbb N}B_n^k)\dpr$ are orthogonal, for every $k\in R$. For $k\in K\setminus R$, let $C_k=D_k=\mathbb C1$. If $n\in\mathbb N$, then $A_n\subset\prod_{\mathcal U}A_n^k\subset\prod_{\mathcal U}C_k$ and $B_n\subset\prod_{\mathcal U}B_n^k\subset\prod_{\mathcal U}D_k$. As this holds for every $n\in\mathbb N$, we get that $A\subset\prod_{\mathcal U}C_k$ and $B\subset\prod_{\mathcal U}D_k$. This finishes the proof of Case 2 and of the theorem. 
\end{proof}

\subsection{Proof of Theorem \ref{inductive}} In order to construct a ${\rm II_1}$ factor satisfying the hypothesis of Theorem \ref{inductive}, we follow closely the construction from \cite[Definition 5.1]{CIKE22}.
This construction uses the following key result from \cite{CIKE22}.

\begin{cor}[Corollary 4.3 in \cite{CIKE22}]\label{amalgam}
Let $(M,\tau)$ be a tracial von Neumann algebra having no type I direct summand. Let $u_1,u_2\in\mathscr U(M)$ such that $\{u_1\}\dpr \perp\{u_2\}\dpr$. 

 Then there exists a ${\rm II_1}$ factor $P=\Phi(M,u_1,u_2)\dpr$ generated by a copy of $M$
and Haar unitaries $v_1,v_2\in \mathscr U(P)$ so that $[u_1,v_1]=[u_2,v_2]=[v_1,v_2]=0$. Moreover, if $Q\subset M$ is a von Neumann subalgebra such that $Q\npreceq_{M}\{u_i\}\dpr$, for every $1\leq i\leq 2$, then $Q'\cap P\subset M$. 

\end{cor}

For a ${\rm II_1}$ factor $M$, we let $\mathscr W(M)$ be the set of pairs $(u_1,u_2)\in\mathscr U(M)\times\mathscr U(M)$ such that  $\{u_1\}\dpr$ and $\{u_2\}\dpr$ are orthogonal. We endow $\mathscr U(M)\times\mathscr U(M)$ with the product $\|\cdot\|_2$-topology. 
We next repeat the construction from \cite[Definition 5.1]{CIKE22} where we replace $\mathscr V(M)$ (the set of pairs $(u_1,u_2)\in\mathscr W(M)$ such that  $u_1^2=u_2^3=1$) with $\mathscr W(M)$.

\begin{df}\label{M}Let $M_1$ be a ${\rm II_1}$ factor. 
We construct a ${\rm II_1}$ factor $M$ which contains $M_1$ and arises as the inductive limit of 
 an increasing sequence $(M_n)_{n\in\mathbb N}$ of ${\rm II_1}$ factors. 
 To this end, let $\sigma=(\sigma_1,\sigma_2):\mathbb N\rightarrow\mathbb N\times\mathbb N$ be a bijection such that $\sigma_1(n)\leq n$, for every $n\in\mathbb N$. 
 Assume that $M_1,\ldots,M_n$ have been constructed, for some $n\in\mathbb N$.
Let $\{(u_1^{n,k},u_2^{n,k})\}_{k\in\mathbb N}\subset\mathscr W(M_n)$ be a $\|\cdot\|_2$-dense sequence.
Since $\sigma_1(n)\leq n$, we have $(u_1^{\sigma(n)},u_2^{\sigma(n)})\in\mathscr W(M_n)$ and we can define $M_{n+1}:=\Phi(M_n,u_1^{\sigma(n)},u_2^{\sigma(n)}).$
Then $M_n\subset M_{n+1}$ and 
$M_{n+1}$ is a ${\rm II_1}$ factor by Corollary \ref{amalgam}. Thus, $M:=({\bigcup_{n\in\mathbb N}M_n})\dpr$ a ${\rm II_1}$ factor.
\end{df}

\begin{prop}\label{2unitaries} Let $M$ be the ${\rm II_1}$ factor introduced in Definition \ref{M} and $\mathcal U$ be a countably cofinal ultrafilter on a set $I$. Let $u_1,u_2\in\mathscr U(M^{\mathcal U})$ such that $\{u_1\}\dpr$ and $\{u_2\}\dpr$ are $2$-independent.

Then there exist Haar unitaries $v_1,v_2\in M^{\mathcal U}$ so that $[u_1,v_1]=[u_2,v_2]=[v_1,v_2]=0$.
\end{prop}

Proposition \ref{2unitaries} follows by repeating the argument used in the proof of \cite[Proposition 5.3]{CIKE22}, which we recall for the reader's convenience.
\begin{proof} 
Since $M=(\bigcup_{n\in\mathbb N}M_n)\dpr$ and $\mathcal U$ is countably cofinal, by applying \cite[Lemma 2.2]{BCI15} we can find ${(n_i)}_{i\in I}\subset \mathbb N$ such that $u_1,u_2\in\prod_{i\in\mathcal U}M_{n_i}$. Also, the proof of \cite[Lemma 2.2]{BCI15} provides a function $f:I\rightarrow\mathbb N$ such that $\lim_{i\rightarrow\mathcal U}f(i)=+\infty$.

Since $\{u_1\}\dpr$ and $\{u_2\}\dpr$ are $2$-independent, Theorem \ref{lifting} provides orthogonal von Neumann subalgebras $C_i,D_i\subset M_{n_i}$, for every $i\in I$, such that $u_1\in\prod_{\mathcal U}C_i$ and $u_2\in\prod_{\mathcal U}D_i$. Thus, we can represent $u_1=(u_{1,i})^{\mathcal U}$ and $u_2=(u_{2,i})^{\mathcal U}$, where $u_{1,i}\in \mathscr U(C_i)$ and $u_{2,i}\in \mathscr U(D_i)$, for every $i\in I$. In particular, $\{u_{1,i}\}\dpr$ and $\{u_{2,i}\}\dpr$ are orthogonal, and thus $(u_{1,i},u_{2,i})\in\mathscr W(M_{n_i})$, for every $i\in I$.

As the sequence $\{(u_1^{n_i,j},u_2^{n_i,j})\}_{j\in\mathbb N}$ is dense in $\mathscr W(M_{n_i})$, we can find $j_i\in\mathbb N$ such that
$\|u_{1,i}-u_1^{n_i,j_i}\|_2+\|u_{2,i}-u_2^{n_i,j_i}\|_2\leq\frac{1}{f(i)}$, for every $i\in I$. For $i\in I$, let $l_i\in\mathbb N$ with $\sigma(l_i)=(n_i,j_i)$. Then  $M_{\sigma(l_i)+1}=\Phi(M_{\sigma(l_i)},u_1^{n_i,j_i},u_2^{n_i,j_i})$. Corollary \ref{amalgam} gives Haar unitaries $v_{1,i},v_{2,i}\in\mathscr U(M_{\sigma(l_i)+1})\subset\mathscr U(M)$ with $[u_1^{n_i,j_i},v_{1,i}]=[u_2^{n_i,j_i},v_{2,i}]=[v_{1,i},v_{2,i}]=0$.
Using that $\lim_{i\rightarrow\mathcal U}f(i)=+\infty$, we conclude that $v_1=(v_{1,i})^{\mathcal U}, v_2=(v_{2,i})^{\mathcal U} \in\mathscr U(M^{\mathcal U})$ are Haar unitaries such that $[u_1,v_1]=[u_2,v_2]=[v_1,v_2]=0$.
\end{proof}

To ensure that $M$ does not have property Gamma, it suffices to take $M_1$ to have property (T), as the next result from \cite{CIKE22} shows:

\begin{prop}[Proposition 5.4 in \cite{CIKE22}]\label{propT}
Assume that $M_1$ has property {\em (T)}. Then $M$ does not have property Gamma.
\end{prop}

\begin{proof}[Proof of Theorem \ref{inductive}]
Let $M_1$ be a separable ${\rm II_1}$ factor with property (T), e.g., take $M_1=\rL(\PSL_n(\mathbb Z))$, for $n\geq 3$.  Let $M$ be constructed as in Definition \ref{M}. The conclusion follows from Propositions \ref{2unitaries} and \ref{propT}.
\end{proof}

\subsection{Proof of Theorem \ref{orthogonal}} We may clearly assume that $z\not=0$ and $z\in M_{\text{sa},1}$, for every $z\in X\cup Y$. Further, we may assume that $X$ and $Y$ consist of pairwise orthogonal vectors. 
Enumerate $X=\{x_1,\dots,x_m\}$ and $Y=\{y_1,\dots,y_n\}$ and define $x=(x_1,\dots,x_m)\in M^m_{\text{sa},1}$ and $y=(y_1,\dots,y_n)\in M^n_{\text{sa},1}$.

By \cite[Corollary 0.2]{Po13a} there exists $v\in\mathscr U(M^{\mathcal U})$ such that $vMv^*$ and $M$ are freely and hence $2$-independent. Then $\|[vx_iv^*,y_j]\|_2=\sqrt{2}\|x_i\|_2\|y_j\|_2>0$ and $\tau^{\mathcal U}(vx_iv^*y_j)=~0$, for every $1\leq i\leq m$, $1\leq j\leq n$. Moreover, for every $(i,j)\not=(i',j')$, we have $\langle [vx_iv^*,y_j], [vx_{i'}v^*,y_{j'}]\rangle=\tau(x_ix_{i'})\tau(y_jy_{j'})=0.$ Thus, we conclude that $\delta(vxv^*,y)>0$ and $\varepsilon(vxv^*,y)=\gamma(vxv^*,y)=0$.
In particular, \begin{equation}\label{condition}13mn\sqrt{\varepsilon(vxv^*,y)}<\delta(vxv^*,y)^2-(mn-1)\gamma(vxv^*,y).\end{equation}

Writing $v=(v_k)^{\mathcal U}$, where $v_k\in\mathscr U(M)$, for all $k\in\mathbb N$. Then  $\lim_{k\rightarrow\mathcal U}\delta(v_kxv_k^*,y)=\delta(vxv^*,y)$, $\lim_{k\rightarrow\mathcal U}\varepsilon(v_kxv_k^*,y)=\varepsilon(vxv^*,y)$ and $\lim_{k\rightarrow\mathcal U}\gamma(v_kxv_k^*,y)=\gamma(vxv^*,y)$. Using \eqref{condition} gives $k\in\mathbb N$ such that $13mn\sqrt{\varepsilon(v_kxv_k^*,y)}<\delta(v_kxv_k^*,y)^2-(mn-1)\gamma(v_kxv_k^*,y)$. By applying Lemma \ref{perturbation}, we can find $w\in\mathscr U(M)$ such that $\varepsilon(w(v_kxv_k^*)w^*,y)=0$. Letting $u=wv_k\in\mathscr U(M)$, we get that $\varepsilon(uXu^*,Y)=0$, i.e., $uXu^*$ and $Y$ are orthogonal.
\hfill$\square$


\bibliographystyle{plain}

\end{document}